\documentclass[12pt]{amsart}

\usepackage{caption,amsmath,amssymb,amsfonts,amsthm,latexsym,graphicx,multirow,color,hyperref,enumerate}
\usepackage[all]{xy}
\usepackage{mathrsfs}
\usepackage{amsmath}
\usepackage{latexsym}
\usepackage{amssymb}
\usepackage{amsfonts}
\usepackage{longtable}
\usepackage{colortbl}
\include{PDF}
\usepackage{graphics}
\usepackage{ulem}
\usepackage{booktabs}
\usepackage{diagbox}
\usepackage{multirow}

\def\ov{\overline} 
\def\l{\langle} \def\r{\rangle}

\def\ZZ{{\sf Z}}

\def\Aut{{\sf Aut}} \def\Inn{{\sf Inn}}\def\Hom{{\sf Hom}}
\def\Out{{\sf Out}}
\def\Cos{{\sf Cos}}

\def\Cay{{\sf Cay}} 
\def\D{{\rm D}} \def\Q{{\rm Q}}
\def\S{{\rm S}}

\def\C{{\bf C}}\def\N{{\bf N}}\def\Z{{\bf Z}}
\def\O{{\bf O}}

\def\Ga{{\it \Gamma}}

\def\a{\alpha}   \def\s{\sigma}

 \def\GL{{\rm GL}}

\def\A{{\rm A}}

\def\PSL{{\rm PSL}}  \def\PGL{{\rm PGL}}
\def\GL{{\rm GL}} \def\SL{{\rm SL}}

  \def\D{{\rm D}}

\def\calM{{\mathcal M}}

\def\Fit{{\rm Fit}}

\def\vs{\vskip0.03in}
\def\le{\leqslant}
\def\ge{\geqslant}

\newtheorem{theorem}{Theorem}[section]%
\newtheorem{lemma}[theorem]{Lemma}%
\newtheorem{corollary}[theorem]{Corollary}%
\newtheorem{proposition}[theorem]{Proposition}%
\newtheorem{construction}[theorem]{Construction}%
\newtheorem{hypothesis}[theorem]{Hypothesis}%

\def\qed{{\hfill$\Box$\smallskip}
\medbreak}

\begin{document}

\title[Groups of square-free characteristic]{Finite groups and arc-transitive maps of square-free Euler characteristic}
\thanks{The project was partially supported by the NNSF of China (11931005).}

\author{Peice Hua}
\address{
SUSTech International Center for Mathematics\\
Southern  University of Science and Technology,  Shenzhen 518055,  P.R.China}
\email{huapc@pku.edu.cn}

\author{Cai Heng Li}
\address{Department of Mathematics \\
SUSTech International Center for Mathematics\\
Southern  University of Science and Technology,  Shenzhen 518055,  P.R.China}
\email{lich@sustech.edu.cn}

\author{Jia Ben Zhang}
\address{Department of Mathematics \\
	SUSTech International Center for Mathematics\\
	Southern  University of Science and Technology,  Shenzhen 518055,  P.R.China}
\email{jiabenzhang@qq.com}

\author{Hui Zhou}
\address{School of Mathematics and Statistics \\
	Hainan University,  Haikou 570228, P.R.China}
\email{zhouh06@qq.com}

\begin{abstract}
A characterization is completed for finite groups acting arc-transitively on maps with square-free Euler characteristic, associated with infinite families of regular maps of square-free Euler characteristic presented.
This is based on a classification of finite groups of which each Sylow subgroup has a cyclic or dihedral subgroup of prime index.

\end{abstract}

\maketitle

\date\today

\section{Introduction}

Edge-transitive maps are categorized into fourteen classes according to local structures and local actions in \cite{GW,ST}, among which five are arc-transitive.
The problem of constructing and classifying important classes of such maps with specific Euler characteristic has attracted considerable attention; refer to \cite{CNS}\,-\,\cite{bi-rotary} for the Euler characteristic to be a negative prime or product of two primes.
We aim towards to a classification of arc-transitive maps with square-free Euler characteristic.

Assume that $ G $ is a finite group acting edge-transitively on a map with square-free Euler characteristic, which is said to be {\it of square-free Euler characteristic} for convenience.
Then $G$ has the following restricted property (Lemma~\ref{Sylow-metac}):
\begin{quote}
{\it each Sylow subgroup has a cyclic or dihedral subgroup of prime index.}
\end{quote}

It leads us to studying groups of square-free Euler characteristic, which in fact extends the characterization of {\it almost-Sylow cyclic groups}, namely, each Sylow subgroup of odd order is cyclic, and the Sylow $ 2 $-subgroup either is trivial or contains a cyclic subgroup of index 2; refer to Zassenhaus \cite{ASC3}, Suzuki \cite{ASC1} and Wong \cite{ASC2}.

In previous work \cite{HLZZ}, we characterized non-solvable groups of square-free Euler characteristic, and in \cite{HUAPC}, we determined $ 2 $-groups of square-free Euler characteristic. As the third paper of this series, the purpose of this paper is to present a classification (characterization) of such groups in the solvable case.

As usual, for a group $G$ and a subset $\pi$ of prime divisors of $|G|$, denote by $G_\pi$ a Hall $\pi$-subgroup of $G$, which is a Sylow subgroup if $\pi$ contains a single prime $p$, and denoted by $G_p$.
For a positive integer $n$, denote by $\ZZ_n$ the cyclic group of order $n$, by $\D_{2n}$ or $\Q_{2n}$ the dihedral or generalized quaternion group of order $2n$, and by $ \A_n $ or $ \S_n $ the alternating or symmetric group of degree $ n $, respectively.
A finite group $G$ is a {\it semidirect product} of $A$ by $B$ if $A\lhd G$ and $B<G$ are such that $G/A\cong B$, denoted by $G=A{:}B$. Denote by $A\circ_C B$ the \textit{central product} of $ A $ and $ B $ (find the definition in Preliminaries).
Our first main result is the following theorem.

\begin{theorem}\label{thm:sfEul-gps}
Let $G$ be a finite solvable group such that each Sylow subgroup of $G$ has a cyclic or dihedral subgroup of prime index.
Then{\rm\,:}

{\rm(1)} The Sylow subgroups $ G_2 $ and $ G_p\  (p>2) $  of $ G $ are respectively given in {\rm Proposition \ref{2-gp}} and {\rm Lemma \ref{p-gp}}.

{\rm(2)} Let $ H\le G $ be the largest normal Hall subgroup of odd order. Then, $ H=A{:}B $, where $A$ is abelian, $ B $ is nilpotent and $ A\cap\Z(H)=1$. 

{\rm(3)} $ G=H{:}K $, where $ K $ satisfies one of the following{\rm\,:}
	\begin{itemize}
		\item[(i)] $K=1$ or $G_2\,;$\vs
		\item[(ii)] $ K=G_{\{2,3\}}$ or $ G_{\{2,7\}} $ is one of the groups{\rm\,:}
		\[\ZZ_2^2{:}G_3,\ \ZZ_2\times(\ZZ_2^2{:}G_3),\ \Q_8{:}G_3,\ \ZZ_4\circ_{\ZZ_2}(\Q_8{:}G_3),\,\ZZ_2^3{:}G_7\,;\]
		\item[(iii)] $K=G_{\{2,3\}}$ is one of the groups{\rm\,:}
		\[\mbox{$ X $, $ \ZZ_2\times X $, $ Y $, $ \ZZ_4\circ_{\ZZ_2} Y $, $ W $,}\]
		where $ X,Y,W $ are given in {\rm Table \ref{tab-main}}, $ X $ is homomorphic to $\ZZ_2^2{:}\S_3\cong\S_4$, $ Y $ is homomorphic to $\Q_8{:}\S_3\cong\GL_2(3)$, and $ W $ is homomorphic to $ \ZZ_2^2{:}\ZZ_3\cong\A_4 $, $ \S_4 $, $\ZZ_4\circ\SL_2(3)$ or $\ZZ_4\circ\GL_2(3);$
	    \begin{table}[h]
	    \centering
	    \renewcommand{\arraystretch}{1.5}
	    \caption{$ X,Y,W $, $ \ell\ge 1 $}\label{tab-main}
	    \scalebox{0.85}{
	    \begin{tabular}{|c|c|c|cc|}
	    	\hline
	    	{\rm Row} & $ X $ & $ Y $ & \multicolumn{2}{c|}{$ W $}\\
	    	\hline
	    	
	    	$1$ & $\ZZ_2^2{:}\D_{2.3^{\ell}}$ & $\Q_8{:}\D_{2.3^{\ell}}$ & $ \D_{2.3^{\ell}}\times\A_4$, &
	    	$ (\ZZ_{3^{\ell}}{:}\ZZ_4)\circ(\Q_8{:}\ZZ_3) $ \\
	    			
	    	$2$ &$\ZZ_{3^{\ell}}
	    	\times \S_4$ & $\ZZ_{3^{\ell}}
	    	\times \GL_2(3)$ & $ (\ZZ_2^2{:}\ZZ_{3^{\ell}})\times\D_6$, &
	    	$ (\Q_8{:}\ZZ_{3^{\ell}})\circ(\ZZ_3{:}\ZZ_4) $
	    	\\
	    			
	    	$3$ & $(\ZZ_2^2{:}\D_{2.3^{\ell}})\times\ZZ_3$ & $ (\Q_8{:}\D_{2.3^{\ell}})\times\ZZ_3$ & $\D_{2.3^{\ell}}\times \S_4$, & $ (\ZZ_{3^{\ell}}{:}\ZZ_4)\circ (\Q_8{:}\S_3) $ \\
	    			
	    	$4$ & $(\ZZ_{3^{\ell}}\times \A_4){:}\ZZ_2$ & $(\ZZ_{3^{\ell}}\times (\Q_8{:}\ZZ_3)){:}\ZZ_2$ & $(\ZZ_2^2{:}{\D_{2.3^{\ell}}})\times \D_6$, & $ (\Q_8{:}{\D_{2.3^{\ell}}})\circ(\ZZ_3{:}\ZZ_4) $ \\
	    			
	    	$5$ & $((\ZZ_2^2{:}\ZZ_{3^{\ell}})\times\ZZ_{3}){:}\ZZ_2$ & $((\Q_8{:}\ZZ_{3^{\ell}})\times\ZZ_{3}){:}\ZZ_2$ & $( \ZZ_2^2\times\D_{2.3^{\ell+1}}){:}\ZZ_3$, & $ (\Q_8\circ(\ZZ_{3^{\ell+1}}{:}\ZZ_4) ){:}\ZZ_3 $ \\
	    			
	    	$6$ &$\ZZ_2^2 {:}(\ZZ_{3^{\ell+1}}{:}\ZZ_6)$ & $\Q_8 {:}(\ZZ_{3^{\ell+1}}{:}\ZZ_6)$ & &  \\
	    	\hline
	   \end{tabular}}
	  \end{table}
    \item[\rm(iv)] $K=G_{\{2,3,7\}}=\ZZ_2^3{:}(G_7{:}G_3).$\vs
	\end{itemize}
\end{theorem}

The next theorem shows that arc-transitive automorphism groups of maps with square-free Euler characteristic are more restricted.

\begin{theorem}\label{thm:arctransmap-gps}
Let $\calM$ be a map with square-free Euler characteristic, and let $G\le\Aut\calM$ be solvable and arc-transitive on $ \calM $.
Then, $G=(A{:}B){:}K$, where $A$ is abelian, $ B $ is nilpotent, $ A\cap\Z(A{:}B)=1 $, and one of the following holds{\rm\,:}

\begin{itemize}
	\item[\rm(1)] $ \calM $ is $ G $-arc-transitive, $ K=G_2 $ is $ \ZZ_2 $ or one of the groups given in {\rm Proposition \ref{2-gp-arc}}.\vs
	
	\item[\rm(2)] $ \calM $ is $ G $-regular, $ K=G_{\{2,3\}} $ is one of the groups{\rm\,:} \[\mbox{$ X $, \ $X\times\ZZ_2$, \ $\D_{2.3^{\ell}}\times\S_4$,\  $(\ZZ_2^2{:}\D_{2.3^{\ell}})\times\D_6$,}\] where $X\in\{\ZZ_2^2{:}\D_{2.3^{\ell}},\,(\ZZ_{3^{\ell}}\times \A_4){:}\ZZ_2,\,((\ZZ_2^2{:}\ZZ_{3^{\ell}})\times\ZZ_{3}){:}\ZZ_2\}.$\vs
	
	\item[\rm(3)] $ \calM $ is $ G $-vertex-reversing, $ K=G_{\{2,3\}} $ is either one of the groups given in {\rm (2)}, or one of the groups{\rm\,:} \[\mbox{$ Y $, \ $Y\circ_{\ZZ_2}\ZZ_4$, \ $ (\ZZ_{3^{\ell}}{:}\ZZ_4)\circ(\Q_8{:}\S_3) $,\
	$ (\Q_8{:}\D_{2.3^{\ell}})\circ(\ZZ_3{:}\ZZ_4) $,}\] where $Y\in\{ \Q_8{:}\D_{2.3^{\ell}},\, (\ZZ_{3^{\ell}}\times (\Q_8{:}\ZZ_3)){:}\ZZ_2,\, (\ZZ_{3}\times (\Q_8{:}\ZZ_{3^{\ell}})){:}\ZZ_2\}.$\vs\vs
	
	\item[\rm (4)] $ \calM $ is $ G $-vertex-rotary, and either $ K=G_{\{2,3\}}$ or $G_{\{2,7\}}$ is one of the groups{\rm\,:} \[\mbox{$\ZZ_{2}^2{:}\ZZ_{3^\ell}$,\, $ \ZZ_2\times(\ZZ_{2}^2{:}\ZZ_{3^\ell} )$,\, $\ZZ_4\circ_{\ZZ_2}(\Q_8{:}\ZZ_{3^\ell})$, \,$ \ZZ_2^3{:}\ZZ_{7^\ell} ${\rm\,;}}\]
	or $K=G_{\{2,3\}}$ is one of the groups given in {\rm Theorem \ref{thm:sfEul-gps}~(iii)}, excluding the following four{\rm\,:} \[\mbox{$(\ZZ_{3^{\ell}}\times \A_4){:}\ZZ_2$,\, $(\ZZ_{3}\times (\ZZ_2^2{:}\ZZ_{3^{\ell}})){:}\ZZ_2$,\, $(\ZZ_{3^{\ell}}\times (\Q_8{:}\ZZ_3)){:}\ZZ_2$,\, $(\ZZ_{3}\times (\Q_8{:}\ZZ_{3^{\ell}})){:}\ZZ_2$.}\]
\end{itemize}
\end{theorem}

The proofs of the above two results are respectively given in Section 4\,-\,5 and Section 6. We remark that, Theorem \ref{thm:arctransmap-gps} does not guarantee the existence of arc-transitive maps $ \calM $ corresponding to each of these groups. The existence and construction of such maps is the focus of the next paper in this series. Here, we construct several interesting examples of regular maps with square-free
Euler characteristic in Section 3, from which the following result follows.

\begin{theorem}\label{thm:reg-maps}
There are infinite families of regular maps which have square-free Euler characteristic of form $ -n(n-2) $ or $ -n(n-3) $, with $ n $ an positive integer.
\end{theorem}

\section{Preliminaries}\label{sec:pre}

Let $G$ be a finite group, and let $ H\leqslant G $ be a subgroup of $ G $. Denote by $G'$ the {\it commutator subgroup} of $G$, and by $\N_G(H)$ and $\C_G(H)$ the \textit{normalizer} and the \textit{centralizer} of $H$ in $G$, respectively. If $G$ is solvable, it is well-known that, for any subset $\pi$ of prime divisors of $|G|$, there exist Hall $ \pi $-subgroups of $ G $ with any two of them conjugate in $ G $, and any $ \pi $-subgroup is contained in some Hall $ \pi $-subgroup.

Let $A,B$ be finite groups whose centers $\Z(A),\Z(B)$ have isomorphic subgroups $C_1,C_2$, respectively.
Let $\varphi$ be an isomorphism from $C_1$ to $C_2$, and let $C=\{(c,c^\varphi)\mid c\in C_1\}$. Then $C$ is a normal subgroup of $A\times B$. The factor group $(A\times B)/C$ is called a {\it central product} of $A$ and $B$, denoted by $A\circ_CB$, and sometimes simply by $A\circ B$ if $C$ is equal to $\Z(A)$ or $\Z(B)$.\vs

Let $\Ga,\Sigma$ be two graphs and $\lambda$ be a positive integer. Denote by $\Ga^{(\lambda)}$ the graph with each edge of $\Ga$ replaced by $\lambda$ edges. Denote by $ \Ga\square\Sigma $ the \textit{Cartesian product} of $\Ga$ and $\Sigma$, which has vertex set $ V(\Ga)\times V(\Sigma) $, where $ (x_1,y_1) $ is adjacent to $ (x_2,y_2) $ if and only if $ x_1=y_1 $ and $ x_2 $ is adjacent to $ y_2 $, or $ x_1 $ is adjacent to $ y_1 $ and $ x_2=y_2 $.\vs

It is known (refer to \cite{RotaMap}) that an arc-transitive (multi-) graph $\Ga$ can be expressed as a coset graph\,:
\[\Ga=\Cos(G,H,J),\ \mbox{or}\ \Cos(G,H,HgH)\ \mbox{in the simple graph case},\]
which has vertex set $[G:H]$ and edge set $[G:J]$, such that \begin{itemize}
	\item[-] $Hg_1$ is incident with $Jg_0$ if and only if $g_1g_0^{-1}\in HJ$; or equivalently,\vs
	\item[-] $Hg_1$, $Hg_2$ are adjacent if and only if $ g_2g_1^{-1}\in HgH $ in the simple graph case.
\end{itemize}
Further, $ \Ga $ is of valency $|H:H\cap H^g|$, where $g\in J\setminus (H\cap J)$, and $ \Ga $ is connected if and only if $ G=\l H,J\r $.\vs

A \textit{map} $ \calM=(V,E,F) $ is a $ 2 $-cell embedding of a graph into a closed surface, with vertex set $V$, edge set $E$ and face set $F$. The {\it Euler characteristic} of $\calM$ is given by Euler formula\,:
 \[\chi(\calM)=|V|-|E|+|F|.\]

A \textit{flag} of $\calM$ is an incidence triple $(\a,e,f)$ with $\a\in V$, $e\in E$ and $f\in F$. Let $ G\le\Aut(\calM) $, and let $ (\a,e,f) $ be a flag. It is well-known that, each of the stabilizers $G_\a$ and $G_f$ is cyclic or dihedral, and meanwhile, $G_e=1$, $\ZZ_2$ or $\ZZ_2^2$.\vs

A map $ \calM $ is called a \textit{$ G $-arc-transitive} map, if $ G $ acts transitively on arcs (directed
edges) of $ \calM $. Arc-transitive maps are divided into five types in \cite{GW,ST}, for which
recent papers \cite{RevMap,RotaMap} provide a new description: letting $ \calM $ be a $ G $-arc-transitive
map, then one of the following holds\,:

(1) $ \calM $ is a \textit{$ G $-regular} map (type 1), which is determined by a \textit{regular triple} $ (x,y,z) $ for $ G $ satisfying\,: \[\mbox{$G=\l x,y,z\r$, $ |x|=|y|=|z|=2 $ and $ \l x,z\r=\ZZ_2^2 $}.\] In this case, $G=\Aut(\calM)$, the stabilizers $G_\a=\l x,y\r,\, G_e=\l x,z\r,\,G_f=\l y,z\r,$ and the underlying graph $\Ga$ is expressed as the coset graph $\Cos(G,H,HJ) $, where $ H=\l x,y\r $, $ J=\l x,z\r $, so $ \Ga $ is connected, of valency $ |H:H\cap H^z| $.\vs

(2) $ \calM $ is a \textit{$ G $-vertex-reversing} map (type $ 2^\ast $ or $ 2^P $), which is determined by a \textit{reversing triple} $(x,y,z)$ for $ G $ satisfying\,: $G=\l x,y,z\r$ and $ |x|=|y|=|z|=2 $.\vs

(3) $ \calM $ is a \textit{$ G $-vertex-rotary} map (type $ 2^*{\rm ex} $ or $ 2^P{\rm ex} $), which is determined by a \textit{rotary pair} $ (\a,z) $ for $ G $ satisfying\,: $G=\l \a,z\r$ and $|z|=2$.\vs\vs

The following lemma established in \cite{HLZZ} characterizes Sylow subgroups of $\Aut(\calM)$ in the case where $\calM$ is with square-free characteristic. Note that, if $\chi(\calM)$ is square-free, then of course so is $\gcd(\chi(\calM),|\Aut\calM|)$.

\begin{lemma}{\rm(\cite[Lem.\,2.2]{HLZZ})}\label{Sylow-metac}
Let $\calM$ be a map such that $\gcd(\chi(\calM),|\Aut(\calM)|)$ is square-free.
Then each Sylow $p$-subgroup of $\Aut(\calM)$ has a cyclic or dihedral subgroup of index $p$.
\end{lemma}

We thus study finite groups satisfying the following hypothesis.

\begin{hypothesis}\label{hypo-0}
{\rm Let $G$ be a finite group such that each Sylow subgroup of $G$ has a cyclic or dihedral subgroup of prime index.}
\end{hypothesis}

The next lemma plays a key role in the analysis of groups satisfying Hypothesis~$\ref{hypo-0}$.

\begin{lemma}{\rm(\cite[Lem.\,2.5]{HLZZ})}\label{class-closed}
Let a finite group $G$ satisfy {\rm Hypothesis~$\ref{hypo-0}$}.
Then so does each subgroup and each factor group of $ G $.
\end{lemma}

\section{Examples of Square-free Euler Characteristic Regular Maps}\label{sec:ex}

We construct three infinite families of regular maps with square-free Euler characteristic in this section.
The constructions are based on the following groups:
\[\begin{array}{lllll}
	X & = & (\l a,s\r\times\l b,t\r){:}\l\s\r\cong \D_{2n}\wr\S_2 &\ \ \ &\\
	& = & \l a,b\r{:}\l s,\s\r\cong \ZZ_n^2{:}\D_8,&\ \ \ &(\ast)
\end{array}\]
where $\l a,s\r=\l a\r{:}\l s\r=\D_{2n}$, $\l b,t\r=\l b\r{:}\l t\r=\D_{2n}$, $(a,s)^\s=(b,t)$ and $|\s|=2$.\vs

The first family is given below.

\begin{construction}\label{cons:sol-1}
{\rm Let $ X $ be the group given in ($ \ast $), and let $G=\l a,s\r\times\l b,t\r=\D_{2n}\times\D_{2n}\le X$. Assume that $n$ is odd. Define $ x,y,z $ as follows\,: \[x=s,\ y=abst,\ z=st.\]
Then, $(x,y,z)$ is a triple of involutions such that $ xz=zx $. Note that, \[\mbox{$ (yx)^2=(abt)^2=a^2 $,\ \  $ (yxz)^2=(abs)^2=b^2 $,}\] and meanwhile, $ \l a\r=\l a^2\r$, $ \l b\r=\l b^2\r$ since $ n $ is odd. Thus, $ a,b,s$ and $t=xz$ are all contained in $\l x,y,z\r,$ so $ G=\l x,y,z\r $, and $(x,y,z)$ is a regular triple for $G$.}

\end{construction}

\begin{lemma}\label{lem:cons-2}
Let $\calM_n=\calM(G,x,y,z)$ be a regular map defined in {\rm Construction~$\ref{cons:sol-1}$}.
Then the underlying graph of $\calM_n$ is a multiple cycle $\C_n^{(n)}$, and $\chi(\calM_n)=-n(n-3)$.
In particular, $ \chi(\calM_n) $ is square-free for infinitely many odd values of $ n $.
\end{lemma}

\begin{proof}
The vertex stabilizer $ H $ is
\[H=\l x,y\r=\l s,abst\r=\l abt\r{:}\l s\r=\l a,s\r\times\l bt\r=\D_{4n},\] and the face stabilizer $ L $ is \[\l y,z\r=\l abst,st\r=\l ab\r{:}\l st\r=\D_{2n}.\] Thus, $(|V|,|E|,|F|)=(\frac{|G|}{|H|},\frac{|G|}{4},\frac{|G|}{|L|})=(n,n^2,2n)$, and the Euler characteristic \[\mbox{$\chi(\calM_n)=n-n^2+2n=-n(n-3)$.}\]

Let $ n=4x+1 $. Then $\chi(\calM_n)=-2(4x+1)(2x-1)$. By a result given by G. Ricci in 1933 (which is improved by $ {\rm P. Erd\ddot{o}s} $ in \cite{Erdos}), there are infinitely many positive integers $ x $ for which $ f(x)=(4x+1)(2x-1) $ is square-free, so there are infinitely many odd values of $ n $ such that $\chi(\calM_n)$ is square-free. Here are some examples\,:
\[\begin{array}{|c|c|c|c|c|c|c|c|} \hline
	n & 5 & 13 & 17 & 29 & 37 &41&\cdots\\
	\hline
	-n(n-3)&-2.5&-2.5.13&-2.7.17&-2.13.29 &-2.17.37 &-2.19.41 &\cdots \\\hline
\end{array}\]

At last, the valency of $ \calM_n $ is  $\frac{2|E|}{|V|}=2n$, and the face length is $\frac{2|E|}{|F|}=n$. Note that, $H\cap H^z=\l a,s\r\cong\D_{2n} $ is of index $ 2 $ in $ H $. Then, the underlying graph $\Ga=\Cos(G,H,HJ)$ is connected of valency 2, so it is the multiple cycle $\C_n^{(n)}$.
\end{proof}

The second family is given below.

\begin{construction}\label{cons:sol-3}
{\rm Let $ X $ be the group given in ($ \ast $), and let $G=\l a,b\r{:}\l st,\s\r=\ZZ_n^2{:}\ZZ_2^2\le X$. Define $x,y,z$ as follows\,: \[x=\s,\ y=ast,\ z=abst.\]
Then, $(x,y,z)$ is a triple of involutions such that $ xz=zx $. Note that, \[\mbox{$ zy=b $,\ \ $ (zy)^x=b^\s=a $.}\] Thus, $ a,b,\s $ and $ st=a^{-1}y $ are all contained in $ \l x,y,z\r $, so $ G=\l x,y,z\r $, and $(x,y,z)$ is a regular triple for $G$.}
\end{construction}

\begin{lemma}\label{lem:cons-3}
	Let $\calM_n=\calM(G,x,y,z)$ be a regular map defined in {\rm Construction~$\ref{cons:sol-3}$}.
	Then the underlying graph of $\calM_n$ is $\C_{n}^{(n)}$, and $\chi(\calM_n)=-n(n-3)$.
	In particular, $ \chi(\calM_n) $ is square-free for infinitely many values of $ n $.
\end{lemma}

\begin{proof}
	The vertex stabilizer $ H $ is \[H=\l x,y\r=\l ast\s\r{:}\l\s\r=\D_{4n},\] and the face stabilizer $ L $ is \[\l y,z\r=\l abst,ast\r=\l b\r{:}\l ast\r=\D_{2n}.\]
	Thus, $(|V|,|E|,|F|)=(\frac{|G|}{|H|},\frac{|G|}{4},\frac{|G|}{|L|})=(n,n^2,2n)$, and the Euler characteristic \[\mbox{$\chi(\calM_n)=n-n^2+2n=-n(n-3)$.}\]
	
	By Lemma \ref{lem:cons-2}, there are infinitely many odd values of $ n $ such that $\chi(\calM_n)$ is square-free. In a similar way, we can prove that there are infinitely many even values of $ n $ with $ n=4x+2 $, such that $\chi(\calM_n)=-2(2x+1)(4x-1)$ is square-free. Here are some examples\,:
	\[\begin{array}{|c|c|c|c|c|c|c|c|} \hline
		n & 2 & 10 & 14 & 22 & 26 & 34 &\cdots\\
		\hline
		-n(n-3)&2&-2.5.7&-2.7.11&-2.11.19 &-2.13.23 &-2.17.31 &\cdots \\\hline
	\end{array}\]
	
	At last, the valency of $ \calM_n $ is  $\frac{2|E|}{|V|}=2n$, and the face length is $\frac{2|E|}{|F|}=n$. Note that, $ H\cap H^z=\l a^{-1}b\r{:}\l\s\r=\D_{2n} $ is of index $ 2 $ in $ H $. Then, the underlying graph $\Ga=\Cos(G,H,HJ)$ is connected of valency 2, so it is the multiple cycle $\C_n^{(n)}$.
\end{proof}

The third family is given below.

\begin{construction}\label{cons:sol-4}
{\rm Let $G=X=\l a,b\r{:}\l s,\s\r= \ZZ_n^2{:}\D_8$ be the group given in ($ \ast $). Assume that $n$ is odd. Define $ x,y,z $ as follows\,: \[x=\s,\ y=s,\ z=abst.\]
Then, $(x,y,z)$ is a triple of involutions such that $ xz=zx $. Note that, \[\mbox{$ (zy)^2=(abt)^2=a^2 $, \ \ $ ((zy)^2)^x=(a^2)^x=b^2 $,}\] and meanwhile,  $ \l a\r=\l a^2\r$, $ \l b\r=\l b^2\r$ since $ n $ is odd. Thus, $ a,b,s,\s $ are all contained in $ \l x,y,z\r $, so $ G=\l x,y,z\r $, and $(x,y,z)$ is a regular triple for $G$.}
\end{construction}

\begin{lemma}\label{lem:cons-4}
Let $\calM_n=\calM(G,x,y,z)$ be a regular map defined in {\rm Construction~$\ref{cons:sol-4}$}.
Then the underlying graph of $\calM_n$ is $\C_n\square\C_n$, and $\chi(\calM_n)=-n(n-2)$.
In particular, $\chi(\calM_n)$ is square-free for infinitely many prime values of $n$.
\end{lemma}

\begin{proof}
The vertex stabilizer $ H $ is \[H=\l x,y\r=\l \s,s\r=\D_8,\] and the face stabilizer $ L $ is \[\l y,z\r=\l s,abst\r=\l abt\r{:}\l s\r=\D_{4n}.\]
Thus, $(|V|,|E|,|F|)=(\frac{|G|}{|H|},\frac{|G|}{4},\frac{|G|}{|L|})=(n^2,2n^2,2n)$, and the Euler characteristic \[\chi(\calM_n)=n^2-2n^2+2n
=-n(n-2).\]

It is similar to Lemma \ref{lem:cons-2} to prove that, there are infinitely many odd values of $ n $ such that $\chi(\calM_n)$ is square-free. Further, by a theorem of Chen \cite{ChenJR}, there indeed exist infinitely many primes $p$ such that $p-2$ is square-free, then so is $-p(p-2)$. Here are some examples\,:
\[\begin{array}{|c|c|c|c|c|c|c|c|c|} \hline
	p & 3 & 5 & 7 & 13 & 17 &19 & 23 &\cdots\\
	\hline
	-p(p-2)&-3&-3.5&-5.7&-11.13 & -3.5.17&-17.19 &-3.7.23 &\cdots \\\hline
\end{array}\]

At last, the valency of $ \calM_n $ is $\frac{2|E|}{|V|}=4$, and the face length is $\frac{2|E|}{|F|}=2n$. Let $\Ga=\Cos(G,H,HJ)$ be the underlying graph of $\calM_n$, and let $ K=\l a,b\r=\ZZ_n^2\leqslant G $. Note that, $G=K{:}H$, so $ K $ is regular on the vertex set of $\Ga$, and $\Ga$ is thus a connected Cayley graph $\Cay(K,S)$ of valency 4.
It follows that $ S=\{x,x^{-1},y,y^{-1}\} $ with $ \l x,y\r=\l a,b\r $, and $\Ga=\Cay(K,S)\cong\C_n\square\C_n$.
\end{proof}

\section{$p$-Groups} \label{sec:p-gps}

In this section, we study $ p $-groups satisfying Hypothesis~\ref{hypo-0}, i.e., $ p $-groups with a cyclic or dihedral subgroup of index $p$. Note that, the case for $ p=2 $ is dealt with in the previous work \cite{HUAPC}, which is introduced below.

Let $\Q_{2^{\ell+1}}=\l u,v\r$ be a {\it generalized quaternion group} of order $2^{\ell+1}$, where \[{\mbox{$|u|=2^\ell\geqslant 2^2$, $u^{2^{\ell-1}}=v^2$, and $u^v=u^{-1}$.}}\] Then $\Q_{2^{\ell+1}}$ has $u^{2^{\ell-1}}=v^2$ as the only involution. Let \[G=\Q_{2^{\ell+1}}\circ\ZZ_4=\l u,v\r\circ \l b\r=(\l u,v\r\times \l b\r)/\l b^2v^2\r.\]
Then $G$ contains normal subgroups $\l u,v\r=\Q_{2^{\ell+1}}$ and $\l b\r=\ZZ_4$, and the only involutions contained in $\l u,v\r$ and $\l b\r$ are $u^{2^{\ell-1}}=v^2$ and $b^2$, respectively, which are actually identical in $G$ as $b^2v^2=1$.

\begin{proposition}{\rm(\cite{HUAPC}, Thm.~1.1)}\label{2-gp}
	Let $G$ be a $2$-group which has a cyclic or dihedral subgroup of index $2$.
	Then one of the following holds. Further, either $ \Aut(G) $ is a $ 2 $-group, or $ (G,\Aut(G))$ is one of $(\ZZ_2^2,\S_3),\, (\ZZ_2^3,\GL_3(2)),\, (\Q_8,\S_4),\, (\Q_{8}\circ \ZZ_4,\S_4\times\ZZ_2) $.
	\begin{itemize}
		\item[\rm (1)] $ G $ is one of the following groups, with a cyclic subgroup of index $ 2 $, $ \ell\geqslant 1 ${\rm :} \vs
		\begin{itemize}
			\item[\rm (1.1)] $ G=\ZZ_{2^{\ell+1}} $ or $ \ZZ_{2^\ell}\times\ZZ_2; $\vs
			\item[\rm (1.2)] $ G=\D_{2^{\ell+2}} $ or $ \Q_{2^{\ell+2}};$\vs
			\item[\rm (1.3)] $ G=\ZZ_{2^{\ell+2}}{:}\ZZ_2=\l a\r{:}\l b\r$ with $ a^b=a^{2^{\ell+1}\pm 1}; $\vs\vs
		\end{itemize}
		
		\item[\rm (2)] $ G $ is one of the following groups, with a dihedral subgroup of index $ 2 $, $ \ell\geqslant 1 ${\rm :} \vs
		\begin{itemize}
			\item[\rm (2.1)] $ G=\D_{2^{\ell+1}}\times\ZZ_2; $\vs
			\item[\rm (2.2)] $G=\D_{2^{\ell+3}}{:}\ZZ_2=\l a,b\r{:}\l c\r$, where $\l a\r{:}\l b\r=\D_{2^{\ell+3}}$, $(a,b)^c=(a^{2^{\ell+1}+1},b);$\vs
			\item[\rm (2.3)] $G=\Q_{2^{\ell+2}}\circ \ZZ_4 $.\vs
		\end{itemize}
	\end{itemize}
\end{proposition}

\begin{proposition}{\rm(\cite{HUAPC}, Thm.~1.1)}\label{2-gp-arc}
	Let $G$ be a $2$-group which has a cyclic or dihedral subgroup of index $2$. Then:
	\begin{itemize}
		\item[\rm (1)] If $ G $ has a reversing triple, then $  G=\D_{2^{\ell+1}} $,  $\D_{2^{\ell+1}}\times\ZZ_2$, $ \D_{2^{\ell+3}}{:}\ZZ_2 $ or $ \Q_{2^{\ell+2}}\circ\ZZ_4 $, where $\ell\geqslant 1;$ further,  if  $ G $ has a regular triple, then $ G\ne \Q_{2^{\ell+2}}\circ\ZZ_4 .$\vs
		\item[\rm (2)] If $ G $ has a rotary pair, then $ G=\ZZ_{2^{\ell+1}}
		 $, $ \ZZ_{2^\ell}\times\ZZ_2
		  $, $
		  \D_{2^{\ell+2}} $ or $
		  \ZZ_{2^{\ell+2}}{:}\ZZ_2 $, $ \ell\ge 1 $.
	\end{itemize}
\end{proposition}

The case for $ p>2$ is dealt with in the following lemma.

\begin{lemma}\label{p-gp}
Let $p>2$ be an odd prime, and let $G$ be a $p$-group which has a cyclic subgroup of index $p$. Then one of the following holds. In particular, each prime divisor of $|\Aut(G)|$ is less than or equal to $p$.
\begin{itemize}
    \item[(1)] $G=\ZZ_{p^\ell}$ is cyclic, and $\Aut(G)=\ZZ_{(p-1)p^{\ell-1}}$.\vs\vs
    \item[(2)] $G=\ZZ_{p^\ell}\times\ZZ_p=\l a,b\r$ is abelian, and
    \begin{enumerate}\vs
    \item[\rm (2.1)] if $\ell>1$, then $ \Aut(G) $ is solvable of order $ p^{\ell+1}(p-1)^2 $, with an abelian Hall $ p' $-subgroup $\Aut(G)_{p'}=(\Aut(\l a\r)\times\Aut(\l b\r))_{p'}=\ZZ_{p-1}\times\ZZ_{p-1};$\vs
    \item[\rm (2.2)] if $\ell=1$, then $\Aut(G)=\GL(2,p)$ is of order $ p(p+1)(p-1)^2 $, with any subgroup $ M\leqslant\GL(2,p) $ of order coprime to $2p$ being abelian.
    \end{enumerate}
    \vs\vs
    \item[(3)] $G=\l a\r{:}\l b\r=\ZZ_{p^\ell}{:}\ZZ_p$ is non-abelian with $ a^b=a^{p^{\ell-1}+1} $, $\ell\geqslant2$, and $\Aut(G)$ is solvable of order $(p-1)p^{\ell+1}$, with a cyclic Hall $ p' $-subgroup $\Aut(G)_{p'}=\l\rho\r=\ZZ_{p-1}$, where $a^\rho=a^\xi\not=a$ and $b^\rho=b$.
\end{itemize}
\end{lemma}

\begin{proof}
According to \cite[5.3.4,~pp.136]{Robinson}, $G$ is one of the groups stated in this lemma. We only need to verify the statements regarding the automorphism groups of $G$.

(i) The automorphism groups of cyclic groups are well-known; refer to \cite[Thm.~7.3]{Rotman}.

(ii) Suppose $ G=\ZZ_{p^\ell}\times\ZZ_p=\l a,b\r $, with $ |a|=p^\ell>p $, $ |b|=p $. By \cite[Thm.~3.2]{Aut-directproduct}, \[
\begin{array}{ccl}
	|\Aut(G)| & = &|\Aut(\ZZ_{p^\ell})|\cdot|\Aut(\ZZ_p)|\cdot|\Hom(\ZZ_{p^\ell},\ZZ_{p})|\cdot|\Hom(\ZZ_{p},\ZZ_{p^\ell})|\\
	& = & (p-1)p^{\ell-1}\cdot (p-1)\cdot p\cdot p\\
	& = & p^{\ell+1}(p-1)^2.
\end{array}\] In particular, $\Aut(\l a\r)\times\Aut(\l b\r)\leqslant \Aut(G)$, so the subgroup \[ (\Aut(\l a\r)\times\Aut(\l b\r))_{p'}\cong(\ZZ_{(p-1)p^{\ell-1}}\times\ZZ_{p-1})_{p'}=\ZZ_{p-1}\times\ZZ_{p-1}\] is an abelian Hall $ p' $-subgroup of $ \Aut(G) $ of order $ (p-1)^2 $.

Let $\Phi(G)=\l a^p\r$ be the Frattini subgroup of $ G $, and let $ \ov G=G/\Phi(G) $. Consider the homomorphism
\[\s:\Aut(G)\rightarrow \Aut(\ov G)=\GL(2,p).\]

On the one hand, the kernel of $ \s $ consists of automorphisms $ g\in\Aut(G) $ which stabilizes each of the $ \Phi(G) $-cosets, and in particular, stabilizes $ \Phi(G)\l a\r $ and $ \Phi(G)\l b\r $. Thus, by direct calculation we have\,: \[a^g\in\{a,a^{p+1},a^{2p+1},\cdots,a^{(p^{\ell-1}-1)p+1}\}\mbox{\ and\ } b^g\in\{b,a^{p^{\ell-1}}b,a^{2p^{\ell-1}}b,\cdots,a^{(p-1)p^{\ell-1}}b\}.\] Consequently, ${\rm Ker}(\s)=\l x\r\times\l y\r\cong \ZZ_{p^{\ell-1}}\times \ZZ_{p},$ where \[ x: a\rightarrow a^{p+1}, b\rightarrow b\ \mbox{\,and\,}\ y:a\rightarrow a, b\rightarrow a^{p^{\ell-1}}b.\]

On the other hand, the image of $ \s $ falls into the (solvable) parabolic subgroup $ P $ of $ \Aut(\ov G)=\GL(2,p) $ which stabilizes $ \l \ov b\r $, because \[\l \ov b\r=\{\ov 1, \ov b,\cdots,\ov b^{p-1}\}=\{\Phi (G), b\Phi (G),\cdots, b^{p-1}\Phi(G)\}\] exactly consists of those elements of $ G $ of order less than $ p^\ell $. Moreover, $ {\rm Im}(\s)=P $, since $ |{\rm Im}(\s)|=|\Aut(G)|/|{\rm Ker}(\s)|=p(p-1)^2=|P| $.

Now, both $ {\rm Ker}(\s) $ and ${\rm Im}(\s) $ are solvable, so is $ \Aut(G) $.

(iii) Suppose $G=\ZZ_{p}\times\ZZ_p$. Then, $\Aut(G)=\GL(2,p)$ is of order $(p^2-1)(p^2-p)=p(p+1)(p-1)^2$, with prime divisors less than or equal to $p$. Let $ M\le\GL(2,p) $ be a subgroup of order coprime to $2p$, and let $ N=M/\Z(M) $. Then, \[ N=M/\Z(M)\leqslant M/(\Z(\GL(2,p))\cap M)\lesssim\PGL(2,p),\] and, in fact, $ N $ is embedded into $ \PSL(2,p) $, because $ 2\nmid |N| $, and meanwhile, $ \PSL(2,p) $ is a index 2 normal subgroup of $ \PGL(2,p) $. The subgroups of $ \PSL(2,p) $ are all explicitly known, by check which a subgroup of order coprime to $ 2p $ must be a cyclic group. Thus, $ N=M/\Z(M) $ is cyclic, so that, $ M $ is abelian.

(iv) Suppose $G=\l a\r{:}\l b\r=\ZZ_{p^\ell}{:}\ZZ_p$ with $ a^b=a^{p^{\ell-1}+1} $, $\ell\geqslant2$. By \cite[Thm.~3.1]{Aut-metac}, $ |\Aut(G)|=(p-1)p^{\ell+1} $, and hence the Sylow subgroup $ \Aut(G)_p $ is normal in $ \Aut(G) $. By \cite[Thm.~3.2]{Aut-metac}, there exists a cyclic Hall $ p' $-subgroup $\Aut(G)_{p'}=\l\rho\r=\ZZ_{p-1}$, where $ \rho $ stabilizes $ \l a\r $ and centralizes $ \l b\r $ as $a^\rho=a^\xi\not=a$ and $b^\rho=b$. Then, $ \Aut(G)=\Aut(G)_p{:}\Aut(G)_{p'} $ is solvable.
\end{proof}

The following corollary will be used in the next Section.

\begin{corollary}\label{lem:3gp}
	Let $ H $ be a $ 3 $-group satisfying {\rm Hypothesis~\ref{hypo-0}}. Then the following statements hold.
	\begin{itemize}
		\item[(1)] Let $ \s\in\Aut(H) $ be an involution. Then $ H{:}\l\s\r$ is isomorphic to one of $\D_{2.3^\ell} $, $ \D_{2.3^\ell}\times\ZZ_3 $, $ \ZZ_{3^\ell}\times\D_6 $, $ (\ZZ_{3^\ell}\times\ZZ_3){:}\ZZ_2 $ and $ \ZZ_{3^\ell}{:}\ZZ_6 $.\vs
		\item[(2)] If $ \Aut(H) $ contains a subgroup $ X $ such that $ X\cong\ZZ_2^2 $, then $ H=\ZZ_{3^{\ell}}\times\ZZ_3 $, and $ H{:}X=\D_{2.3^\ell}\times\D_6 $.
		\item[(3)] The group $ \Out(H) $ contains no subgroups isomorphic to $ \A_4 $.
	\end{itemize}
\end{corollary}

\begin{proof}
	By Lemma \ref{p-gp}, $H$ is one of $ \ZZ_{3^\ell} $, $ \ZZ_{3^\ell}\times\ZZ_3 $ and $ \ZZ_{3^\ell}{:}\ZZ_3 $. Let $ \s\in\Aut(H) $ be an involution.
	
	(i) If $ H=\ZZ_{3^\ell} $, then $ \Aut(H)=\ZZ_{2.3^{\ell-1}} $ contains only one involution, $ H{:}\l\s\r=\D_{2.3^{\ell}} $, and $ \Out(H)=\Aut(H) $ contains no subgroups isomorphic to $ \ZZ_2^2 $ or $ \A_4 $.
	
	(ii) If $ H=\ZZ_{3^\ell}{:}\ZZ_3=\l a\r{:}\l b\r $ is non-abelian as given in Lemma \ref{p-gp}~(3), then $ \Aut(H) $ is solvable of order $ 2.3^{\ell+1} $, so $ \l\s\r $ is a Hall $ 3' $-subgroup of $ \Aut(H) $ of order 2, and it is conjugate to the following one: \[\mbox{$ \l\rho\r=\ZZ_2 $, with $ (a,b)^\rho=(a^{-1},b) $.}\] Thus, $ H{:}\l\s\r\cong H{:}\l\rho\r=\D_{2.3^\ell}{:}\ZZ_3=\ZZ_{3^\ell}{:}\ZZ_6 $. Meanwhile, since $|\Aut(H)|$ is not divisible by 4, $ \Aut(H) $ has no subgroups isomorphic to $ \ZZ_2^2 $, and $ \Out(H) $ has no subgroups isomorphic to $ \A_4 $.
	
	(iii) If $ H=\ZZ_{3}\times\ZZ_3 $, then $ \Aut(H)=\Out(H)$, and $\Aut(H)=\GL_2(3)\cong\Q_8{:}\S_3 $ is solvable, containing one central involution and twelve non-central involutions, with the non-central ones conjugate to the element with diagonal matrix $ (-1,1) $. Hence, $ H{:}\l\s\r$ is isomorphic to $ (\ZZ_{3}\times\ZZ_3){:}\ZZ_2 $ or $ \D_6\times\ZZ_3 $, where in the former case, $ \Z(\GL_2(3))=\ZZ_2 $ acts on $ \ZZ_{3}\times\ZZ_3 $ by mapping each element to its inverse.
	
	Let $ X\le\Aut(H) $ be a subgroup with $ X\cong\ZZ_2^2 $. Then, $ X\cap\Q_8>1 $, so  $ X $ contains the central involution with diagonal matrix $ (-1,-1) $, and $ X $ contains other two non-central involutions conjugate to $ (-1, 1)$ and $(1,-1) $. Hence $ H{:}X=\D_6\times\D_6 $. Moreover, if $ \Aut(H) $ contains a subgroup isomorphic to $ \A_4\cong\ZZ_2^2{:}\ZZ_3 $, then $ \ZZ_3 $ acts on $ \ZZ_2^2 $ by moving the central involution, not possible.\vs
	
	(iv) At last, suppose $ H=\ZZ_{3^\ell}\times\ZZ_3=\l a,b\r $, where $ |a|=3^{\ell}>3 $, $ |b|=3 $. Then, $ \Aut(H)=\Out(H) $. Let $ X\leqslant\Aut(H) $ be a subgroup with $ X\cong\ZZ_2^2 $. By Lemma \ref{p-gp}~(2), $ \Aut(H) $ is solvable of order $ 4.3^{\ell+1} $, so $ X $ is a Hall $ 3' $-subgroup of $ \Aut(H) $ of order $ 4 $, and it is conjugate to the following one: \[ M=(\Aut(\l a\r)\times\Aut(\l b\r))_{3'}=\ZZ_2\times\ZZ_2=\l x,y\r,\] where $(a,b)^x=(a^{-1},b) $ and $ (a,b)^y=(a,b^{-1}) $. Thus, $ H{:}X\cong H{:}M=\D_{2.3^\ell}\times\D_6 $. In particular, the involution $ \s\in\Aut(H) $ is contained in some Hall $ 3' $-subgroup, thus conjugate to $ x $, $ y $ or $ xy $; respectively, $ H{:}\l\s\r$ is isomorphic to $ \D_{2.3^\ell}\times\ZZ_2 $, $ \ZZ_{3^\ell}\times\D_6 $ or $ (\ZZ_{3^\ell}\times\ZZ_3){:}\ZZ_2 $.
	
    At last, let $ g\in\N_{\Aut(H)}(M) $. Note that, $ M $ stabilizes subgroups $ \l a\r $ and $ \l b\r $, and meanwhile it cannot stabilize
	\begin{itemize}
		\item[-] any of the other two subgroups of order $ 3^\ell $, i.e., $ \l ab\r$, $\l ab^{-1}\r $, and
		\item[-] any of the other two non-characteristic subgroups of $ 3 $, i.e., $ \l a^{3^{\ell-1}}b\r $, $ \l a^{-3^{\ell-1}}b\r $.
	\end{itemize}
	Now, $ M=M^g $ stabilizes $ \l a\r^g$ and $\l b\r^g $, which implies $ \l a\r^g=\l a\r $, $ \l b\r^g=\l b\r $, so that, $ \N_{\Aut(H)}(M)\leqslant \Aut(\l a\r)\times\Aut(\l b\r) $ is abelian. Since any subgroup $ Y\leqslant \Aut(H) $ with $ Y\cong \ZZ_2^2 $ is conjugate to $M$, the normalizer $ \N_{\Aut(H)}(Y) $ is conjugate to $ \N_{\Aut(H)}(M) $, so it is abelian, too. Consequently, $ \Aut(H) $ has no subgroup isomorphic to $ \A_4 $.
\end{proof}

\section{Proof of Theorem~\ref{thm:sfEul-gps}}\label{sec:thm-1}

In this section, we give a proof of Theorem~\ref{thm:sfEul-gps}, by studying the way how the $ p $-groups compose a larger solvable group.

Let $ G $ be a finite group, with $ \{p_1,\ldots,p_s\} $ the prime divisors of $ |G| $. For $ p=p_i $, denote by $ \O_p(G) $ the \textit{maximal normal $ p $-subgroup} of $ G $, or equivalently, the intersection of all Sylow $ p $-subgroups of $ G $. Let $ F=\Fit(G) $ be the \textit{Fitting subgroup} of $G$, which is the unique largest nilpotent normal subgroup. Then,
\[F=F_{p_1}\times\cdots\times F_{p_s}, \mbox{\ with\ each\ } F_{p_i}=\O_{p_i}(G).\]  Note that, both $ \O_p(G) $ and $\Fit(G)$ are characteristic in $ G $.
Further, if $ G $ is solvable, then $\C_G(F)\le F$, and so
\[G/F=\N_G(F)/\left(F\C_G(F)\right)\lesssim\Out(F).\]\vs

We first consider solvable groups of odd order.

\begin{lemma}\label{lem:p-normal}
	Let $G$ be a solvable group of odd order satisfying {\rm Hypothesis~$\ref{hypo-0}$}. Let $p$ be the largest prime divisor of $|G|$. Then:
	\begin{itemize}
		\item[\rm (1)] a Sylow p-subgroup $ G_p $ is normal in $ G $, and hence $ G=G_p{:}G_{p'}\, ;$\vs
		\item[\rm (2)] if $G$ is a $\{p,q\}$-group, with q a prime such that $ q<p $ and $q\nmid p^2-1$,
		then $ G=G_p\times G_q $.
	\end{itemize}
\end{lemma}

\begin{proof}
	(1) Let $ \ov G=G/\O_p(G) $. Note that, $ \ov G $ has no normal $ p $-subgroup. Thus, $ F:=\Fit(\ov G)=F_{p_1}\times\cdots\times F_{p_s} $ with all $ p_i<p $, and \[\ov G/F\lesssim \Out(F)=\Out(F_{p_1})\times\cdots\times\Out(F_{p_s}).\] By Lemma \ref{p-gp}, $ p\nmid |\Out(F_{p_i})| $ as $ p>p_i $. Then, $ p\nmid |\ov G/F| $. It follows that $ p\nmid |\ov G| $, so $ G_p=\O_p(G)\lhd G $ and $ G=G_p{:}G_{p'} $.
	
	(2) Suppose that $ G $ is a $ {\{p,q\}} $-group with $ p>q $ and $q\nmid p^2-1$. By (1), we have $G=G_p{:}G_q$. Meanwhile, $ H:=G/\O_q(G) $ has no normal $ q $-subgroup, so $ \Fit(H) $ is a $ p $-group, and $ H/\Fit(H)\lesssim\Out(\Fit(H))$. By Lemma \ref{p-gp}, any prime divisor of $ |\Out(\Fit(H))| $ is a divisor of $ p(p^2-1) $. Then, $ q\nmid|H/\Fit(H)| $ as $q\nmid p(p^2-1)$. It follows that $ q\nmid|H|$, so $ G_q=\O_q(G)\lhd G $ and $ G=G_p\times G_q $.
\end{proof}

Before giving a general characterization of solvable groups of odd order, we need the following fact.

\begin{lemma}\label{lem0}
	Let $ H=\ZZ_{p^\ell} $, $ 1\ne x\in\Aut(H) $ be an automorphism with $ \gcd(|x|,p)=1 $, and $ X=H{:}\l x\r $. Then, {\rm (i)} $ H\leqslant X'\,; $ {\rm (ii)} $ H\cap \Z(X)=1 $.

\end{lemma}

\begin{proof}
Let $ H=\l a\r $. By Lemma \ref{p-gp}~(1), $ |x| $ divides $ p-1 $, and thus we can suppose \[\mbox{$ a^x=a^m $,\ where $ 1<m\leqslant p^\ell-1 $,\, $p\nmid m$\, and\, $ p^\ell\mid m^{p-1}-1 $.}\] If $ p\mid m-1 $, then by {\sf LTE} (Lift the Exponent Lemma), we have $ p^\ell\mid m-1 $, so $ a^x=a^m=a $ and $ x=1 $, a contradiction. Hence $ p\nmid m-1 $. Thus, (i) the commutator $ [a,x]=a^{-1}a^x=a^{m-1}$ can generate $ H $, so that, $ H\leqslant X' $\,; (ii) for any $ 1\leqslant k<p^\ell $, $ (a^k)^x=a^{km}\ne a^k $, so that, $ H\cap \Z(X)=1 $.
\end{proof}

\begin{lemma}\label{lem:abel-G_2}
Let $G$ be a solvable group of odd order satisfying {\rm Hypothesis~$\ref{hypo-0}$}.
Then $G=A{:}B$, where $A$ is abelian, $B$ is nilpotent, and
$A\cap \Z(G)=1$.
\end{lemma}

\begin{proof}

The proof proceeds by induction on $ n $, the number of prime divisors of $ |G| $. For $ n=1 $, the lemma holds with $ A=1 $, $ B=G $.
Suppose $n\geqslant2$, and assume inductively that the lemma holds for groups of order with less prime divisors.

Let $ p $ be the largest prime divisor of $ |G| $.
By Lemma~\ref{lem:p-normal} and the induction assumption, we have $$ G=G_p{:}G_{p'}=G_p{:}(C{:}D), $$ with $ C $ abelian, $ D $ nilpotent, and $C\cap \Z(C{:}D)=1$, so that, $ D $ does not
centralize any Sylow subgroup of $ C $.

If $G=G_p\times G_{p'}$, then
$$G=G_p{\times}(C{:}D)=(G_p{\times}C){:}D=C{:}(G_p\times D),$$
which gives an expected factorization of $ G $, with $ A=C $, $ B=G_p\times D $.

We thus assume that $G\not=G_p\times G_{p'}$, namely, $G_{p'}$ acts by conjugation on $G_p$ non-trivially. We shall confirm the following claim.

{\bf Claim:} $G_p$ is centralized by each Sylow $r$-subgroup $C_r$ of $C$, so that, $G_p$ is centralized by $ C $.

By Lemma~\ref{p-gp}, we have two cases to treat: $C_r=\ZZ_{r^k}$ or $\ZZ_{r^k}\times\ZZ_r$.

{\bf Case 1.}\ Let $ C_r=\ZZ_{r^k}$. If $[C_r,D_{r'}]=1$, since $ [C_r,D]\ne 1 $, then $ (C{:}D)_r=C_r{:}D_r=\ZZ_{r^k}{:}\ZZ_{r} $ is non-abelian as given in Lemma \ref{p-gp}~(3), and so $$1\not=C_r\cap \Z(C_r{:}D_r)=C_r\cap\Z(C{:}D) \leqslant C\cap\Z(C{:}D),$$ a contradiction.
Thus, there exists a Sylow $q$-subgroup $D_q$ with $q\not=r$ which does not centralize $C_r$. Let $ x \in D_q$ act non-trivially on $ C_r $, and let \[ X=C_r{:}\l x\r.\]
Then, $ X $ acts on $G_p$ with $ \ov X=X/\C_{X}(G_p) \lesssim \Aut(G_p)$. Note that, $ |X| $ is coprime to $ 2p $, so is $ |\ov X| $. By Lemma~\ref{p-gp}, either $ \Aut(G_p)=\GL_2(p) $, or $ \Aut(G_p) $ is solvable with an abelian Hall $ p' $-subgroup, while both cases imply that $ \ov X $ is abelian, that is, $ X'\leqslant \C_{X}(G_p) $. At last, by Lemma \ref{lem0}~(i), we have $ C_r\leqslant X' $, and so $ C_r\leqslant\C_{X}(G_p) $.

{\bf Case 2.}\ Let $ C_r=\ZZ_{r^k}\times\ZZ_{r}$, where $k\geqslant1$.
In this case, $ C_r=(C{:}D)_r $, namely, $r$ does not divide $|D|$.
It follows by Lemma~\ref{p-gp}~(2) that, either
\begin{itemize}
	\item[-] $C_r=\ZZ_r^2$ and $D$ is irreducible on $C_r$; or\vs
	\item[-] there exists $a,b\in C_r$ such that $C_r=\l a\r\times \l b\r$, and both $\l a\r$ and $\l b\r$ are normalized by $D$.
\end{itemize}

For the former, let $ x\in D$ act non-trivially on $ \ZZ_r^2 $, and let $X=\ZZ_r^2{:}\l x\r$, which has $ \ZZ_r^2 $ as a minimal normal subgroup; for the latter, there exist $u,v\in D$ respectively acts on $\l a\r$, $\l b\r$ non-trivially, so let $ Y=\l a\r{:}\l u\r $, $ Z=\l b\r{:}\l v\r $. Note that, each of $ X,Y,Z $ induces some automorphism group of $ G_p $ of order coprime to $ 2p $. It is similar to {\bf Case 1} to show $[G_p,C_r]=1$.

We have therefore confirmed the claim. Consequently, we have
\[G=G_p{:}(C{:}D)=(G_p\times C){:}D.\]
We then find an expected factorization of $ G $. By assumption, $ G_{p'}=C{:}D $ acts by conjugation on $G_p$ non-trivially. Now, as $ C $ acts on $ G_p $ trivially, $ D $ acts on $G_p$ non-trivially, so it induces some automorphism of $ G_p $ of order coprime to $ p $.

By Lemma \ref{p-gp}, we have three cases to treat:

(1) Suppose $ G_p=\ZZ_{p^\ell} $. By Lemma \ref{lem0}~(ii), any non-identity element of $ G_p $ cannot fixed by $ D $. Thus $G=A{:}B$, with $A=G_p\times C$, $B=D$, satisfies the theorem.

(2) Suppose $ G_p=\ZZ_{p^\ell}{:}\ZZ_{p} $ as given in Lemma \ref{p-gp}~(3). Then there exist $ a,b\in G_p $ such that $ G_p=\l a\r{:}\l b\r $, $ \l a\r $ is normalized by $ D $, and $ \l b\r $ is centralized by $ D $. Rewrite $ G $ as
\[G=(C\times(\l a\r{:}\l b\r)){:}D=(C\times\l a\r){:}(D\times \l b\r).\]
By Lemma \ref{lem0}~(ii), $G=A{:}B$, with $A=C\times \l a\r$, $B=D\times\l b\r$, satisfies the theorem.

(3) Suppose $ G_p=\ZZ_{p^\ell}\times\ZZ_{p} $. It follows by Lemma~\ref{p-gp}~(2) that, either $G_p=\ZZ_p^2$ and $D$ is irreducible on $G_p$; or there exist $a,b\in G_p$ such that $G_p=\l a\r\times \l b\r$, and both $\l a\r$, $\l b\r$ are normalized by $D$. For the former, $G=A{:}B$ satisfies the theorem, with $A=G_p\times C$, $B=D$. For the latter, if $ D $ acts non-trivially on both $ \l a\r $ and $ \l b\r $, then also let $A=G_p\times C$, $B=D$; otherwise, $ D $ acts non-trivially on some one of $ \l a\r $ and $ \l b\r $, say $ \l a\r $, so let $A=C\times \l a\r$, $B=D\times\l b\r$.
\end{proof}

We next deal with solvable groups of even order.

\begin{proposition}\label{2p-gp}
	Let $K$ be a $\{2,p\}$-group with $p$ odd prime satisfying {\rm Hypotheses~$\ref{hypo-0}$}.
	Assume that $K\ne K_p{:}K_2$. Then $ p\in\{3,7\} $ and one of the following holds.
	\begin{enumerate}
		\item[\rm (1)] $K=\ZZ_2^3{:}K_7$,\, $\ZZ_2^2{:}K_3 $,\, $ \ZZ_2\times(\ZZ_2^2{:}K_3) $,\, $ \Q_8{:}K_3$ or $\ZZ_4\circ_{\ZZ_2} (\Q_8{:}K_3)\,;$\vs
		\item[\rm (2)] $K$ is one of the $ \{2,3\} $-groups: $ X $, $ \ZZ_2\times X $, $ Y $, $ \ZZ_4\circ_{\ZZ_2} Y $ and $ W $, where $ X,Y,W $ are given in {\rm Table \ref{tab-main}}, $ X $ is homomorphic to $\ZZ_2^2{:}\S_3\cong\S_4$, $ Y $ is homomorphic to $\Q_8{:}\S_3\cong\GL_2(3)$, and $ W $ is homomorphic to $ \ZZ_2^2{:}\ZZ_3\cong\A_4 $, $ \S_4 $, $\ZZ_4\circ\SL_2(3)$ or $\ZZ_4\circ\GL_2(3)$.
	\end{enumerate}
\end{proposition}

We remark that, the proof of Proposition \ref{2p-gp} has been moved to the end of this Section, as it is rather lengthy.

\begin{lemma}\label{lem:G_{237}}
	Let $K$ be a $\{2,3,7\}$-group satisfying {\rm Hypotheses~$\ref{hypo-0}$}, with neither $ K_3 $ nor $ K_7 $ normal in $ K $.
	Then $K=K_2{:}(K_7{:}K_3)=\ZZ_2^3{:}(K_7{:}K_3)$. 
\end{lemma}

\begin{proof}
	By Lemma \ref{lem:p-normal}, $ K_{\{3,7\}}= K_{7}{:}K_{3} $. If $ K_{\{2,7\}}= K_{7}{:}K_{2} $, then $ K_7 $ is normal in $ K $, not in our case. Thus, $ K_{\{2,7\}}\ne K_{7}{:}K_{2} $, and by Proposition \ref{2p-gp}, we have \[ K_{\{2,7\}}=K_2{:}K_7=\ZZ_2^3{:}K_7.\]
	
	We claim that, $ K_2=\ZZ_2^3 $ centralizes the Fitting group $ F=\Fit(K)=F_2\times F_3\times F_7 $. If this is true, then we are done, since $ K_2\leqslant C_{K}(F)\leqslant F $, so  $ K_2=F_2\lhd K $, and \[K=K_2{:}K_{\{3,7\}}=K_2{:}(K_7{:}K_3)=\ZZ_2^3{:}(K_7{:}K_3).\]
	
	It follows by $ F_2\le K_2=\ZZ_2^3 $ that $ K_2 $ centralizes $ F_2 $. Meanwhile, $ K_2 $ is normalized by $ K_7 $, so $ K_2 $ and $ F_7 $ normalizes, and hence centralizes, each other. The claim is then proved if $ K_2 $ centralizes $ F_3 $. Suppose that $ K_2 $ doe not centralize $ F_3 $. Consider the subgroup $ F_3{:}(K_2{:}K_7)\leqslant K $. Then $ K_2=\ZZ_2^3 $ as a minimal normal subgroup of $ K_2{:}K_7 $ is embedded into  $\Aut(F_3) $. By Lemma \ref{p-gp}, $ |\Aut(F_3)|$ being divisible by 8 only happens when $ F_3=\ZZ_3^2 $, $\Aut(F_3)=\GL_2(3)\cong\Q_8{:}\S_3 $. However, $ \GL_2(3) $ has no subgroup isomorphic to $ \ZZ_2^3 $, a contradiction.
\end{proof}

Now, we are ready to prove Theorem~\ref{thm:sfEul-gps}.

{\bf Proof of Theorem~\ref{thm:sfEul-gps}\,:}

Let $G$ be a finite solvable group of which each Sylow subgroup has a cyclic or dihedral subgroup of prime index. Then, the Sylow subgroups $ G_2 $ and $ G_p $ ($ p>2 $) are respectively given in Proposition \ref{2-gp} and Lemma \ref{p-gp}.

Let $H\le G$ be the largest normal Hall subgroup of odd order, with $\pi=\pi(H)$.
Let $K$ be a Hall $\pi'$-subgroup of $G$, where $\pi'=\pi(G)\setminus\pi(H)$.
Then, we have\,: $$G=H{:}K.$$

Since $ H $ is of odd order, by Lemma~\ref{lem:abel-G_2}, $H=A{:}B$ such that $ A $ is abelian, $B$ is nilpotent, and $A\cap\Z(H)=1$. In particular, if $ G $ is of odd order, the $ G=H $, $ K=1 $.

We then assume that $ G $ is of even order, so $ 2\in\pi' $. For any prime divisor $ p $ of $ |G| $ with $ p>7 $ or $ p=5 $, by Lemma \ref{lem:p-normal}, $ G_p $ is normalized by $ G_3 $ and $ G_7 $, while by Proposition \ref{2p-gp},  $ G_p $ is normalized by $ G_2 $. Thus, $ G=G_{\{2,3,7\}'}{:}G_{\{2,3,7\}} $. Note that, $ HG_{\{2,3,7\}'}=G_{\pi\,\cup\,\{2,3,7\}'} $ is also a normal Hall subgroup of odd order, so by our assumption, $\{2,3,7\}'\subseteq\pi$. That is, $\pi'\subseteq\{2,3,7\}$, and one of the following holds\,:
\begin{itemize}
	\item[(i)] $\pi'=\{2\}$, $ K=G_2 $\,;\vs\vs
	\item[(ii)] $\pi'=\{2,3\}$ or $\{2,7\}$, where, respectively, $ K=G_{\{2,3\}}\ne G_{3}{:}G_2 $ or $ K=G_{\{2,7\}}\ne G_{7}{:}G_2 $, so $ K $ is given in Proposition \ref{2p-gp}\,;\vs\vs
	\item[(iii)] $\pi'=\{2,3,7\}$, $ K=G_{\{2,3,7\}}$ with neither $ G_3 $ nor $ G_7 $ normal in $ K $, so $ K $ is given in Lemma \ref{lem:G_{237}}.
\end{itemize}

Then, the proof is complete.\qed
\vs\vs

{\bf A proof of Proposition \ref{2p-gp}} is given in the remaining part of this section\,:

Let $K$ be a $\{2,p\}$-group with $p$ odd prime  satisfying {\rm Hypotheses~$\ref{hypo-0}$}. Let $ F=\Fit(K)=F_2\times F_p $. Assume that $ K_p $ is not normal in $ K $, so $K\ne K_p{:}K_2$, and $ F_p<K_p $.\vs

We first prove $ p\in\{3,7\} $ and determine the group $ \ov K:=K/F_p $.

The group $ \ov K $ is also a $ \{2,p\} $-group as $ F_p<K_p $, and meanwhile, $ \ov K $ has no normal $ p $-subgroups, so $\Fit(\ov K)=\O_2(\ov K)$, and \[H:=\ov K/\Fit(\ov K)\lesssim\Out(\Fit(\ov K))=\Out(\O_2(\ov K)).\] Now, $ |\ov K| $ is divisible by $ p $, so are $ |H| $ and $ |\Out(\O_2(\ov K))| $. Thus, by Proposition~\ref{2-gp}, the triple $ (\O_2(\ov K),\Aut(\O_2(\ov K)),\Out(\O_2(\ov K))) $ is one of the following~: \[\mbox{$(\ZZ_2^2,\S_3,\S_3),\, (\ZZ_2^3,\GL_3(2),\GL_3(2)),\, (\Q_8,\S_4,\S_3),\, (\Q_{8}\circ \ZZ_4,\S_4\times\ZZ_2,\S_3\times\ZZ_2) $,}\] where $ \Out(\O_2(\ov K))=\Aut(\O_2(\ov K))/\Inn(\O_2(\ov K))\lesssim\Aut(\O_2(\ov K)) $ can be regarded as a subgroup.

Note that, $ H $ is a subgroup of $ \Out(\O_2(\ov K)) $, and meanwhile, $H=\ov K/\O_2(\ov K) $, so it has no normal 2-subgroups. By checking the subgroups of $ \Out(\O_2(\ov K)) $, we conclude that, either $p=7$, $ H=\ZZ_7 $, $ \ov K=\ZZ_2^3{:}\ZZ_7 $; or $ p=3 $, $ H\in\{\ZZ_3,\S_3\} $, and one of the cases in Table \ref{tab-ovK} holds.

\begin{table}[h]
	\centering
	\caption{$ \ov K=K/F_3 $}\label{tab-ovK}
	\renewcommand{\arraystretch}{1.5}
	\scalebox{0.85}{
		\begin{tabular}{|c|c|c|c|c|}
			\hline
			$ \O_2(\ov K) $ & $\ZZ_2^2 $ & $\ZZ_2^3 $ & $ \Q_8 $ & $ \ZZ_4\circ \Q_8 $\\
			\hline
			$ \ov K=\O_2(\ov K){:}\ZZ_3 $ & $\A_4 $ & $ \ZZ_2\times \A_4 $ & $ \SL_2(3) $ & $ \ZZ_4\circ\SL_2(3) $ \\
			\hline
			$ \ov K=\O_2(\ov K){:}\S_3 $ & $\S_4 $ & $ \ZZ_2\times \S_4 $ & $ \GL_2(3) $ & $ \ZZ_4\circ\GL_2(3) $ \\
			\hline
	\end{tabular}}
\end{table}

We remark that, the groups in Table \ref{tab-ovK} satisfy Hypothesis \ref{hypo-0}. In the case $ H=\ZZ_3 $, this is natural as the Sylow 2-subgroup $\ov K_2=\O_2(\ov K) $ comes from Proposition~\ref{2-gp}. Meanwhile, in the case $ H=\S_3 $, we have $\ov K_2=\O_2(\ov K){:}\ZZ_2$, and respectively for $ \O_2(\ov K)=\ZZ_2^2 $, $ \ZZ_2^3 $, $ \Q_8 $ or $ \ZZ_4\circ \Q_8 $, it is verified by simple calculation that $ \ov K_2 $ has a index 2 subgroup isomorphic to $ \ZZ_4 $, $ \D_8 $, $ \ZZ_8 $ or $ \D_{16} $.\vs\vs

For convenience, we then consider $ K $ in three separate cases. Note that, $ F_2\lhd K $, $ F_2\cong\ov {F_2}\lhd\ov K $, so $ F_2\lesssim\O_2(\ov K) $. We thus have cases as follows\,:
\[\mbox{{\bf(I)} $ F_2=\O_2(\ov K) $, $ H\in\{\ZZ_7,\ZZ_3\} $\,;\ \ \ {\bf (II)} $ F_2=\O_2(\ov K) $, $ H=\S_3 $\,;\ \ \ {\bf (III)} $ F_2<\O_2(\ov K). $}\]

Note that, {\bf (I)} exactly corresponds to the case $ F_2=K_2\lhd K$. In this case, either $ p=7 $, $ K=\ZZ_2^3{:}K_7 $ is with $ K/F_7 \cong\ZZ_2^3{:}\ZZ_7 $; or $ p=3 $, $ K_2\in\{\ZZ_2^2, \ZZ_2^3,\Q_8, \ZZ_4\circ \Q_8\} $, and $ K=K_2{:}K_3 $ is with $ K/F_3 $ equal to one of the groups in Table \ref{tab-ovK}, so $ K $ is one of
\[\ZZ_2^2{:}K_3,\ \ZZ_2\times(\ZZ_2^2{:}K_3),\ \Q_8{:}K_3,\ \ZZ_4\circ_{\ZZ_2} (\Q_8{:}K_3),\] where $ \ZZ_4\circ_{\ZZ_2} (\Q_8{:}K_3) $ cannot be simply written as $ \ZZ_4\circ (\Q_8{:}K_3) $, as we might have $ \Z(\Q_8{:}K_3)>\ZZ_2 $.
These groups are given in {\bf Proposition \ref{2p-gp}~(1)}.\vs\vs

We then dealt with {\bf (II)} and {\bf (III)} in Lemmas \ref{II} and \ref{III-1}\,-\,\ref{III-3}, respectively.

\begin{lemma}\label{II}
	If it is in case {\bf(II)}, then $ K $ is one of the groups{\rm\,:} $ X $, $ \ZZ_2\times X $, $ Y $, $ \ZZ_4\circ_{\ZZ_2} Y $, where $ X,Y $ are given in {\rm Table \ref{tab-main}}, $ X $ is homomorphic to $\S_4$, and $ Y $ is homomorphic to $\GL_2(3)$.
\end{lemma}

\begin{proof}
In case {\bf(II)}, we have\,: \[\mbox{$ F_2=\O_2(\ov K)\in\{\ZZ_2^2,\, \ZZ_2^3,\, \Q_8,\, \ZZ_4\circ \Q_8\} $, $ H=\S_3 $, and $ K=F_3.\ov K=F_3.(F_2{:}\S_3) $.}\] Let $ L\le K $ be the full preimage of $\S_3\le\ov K$. Then, $ \l F_2,L\r=F_2L $ is the full preimage of $\ov K$, so $ F_2L=K $. Meanwhile, $ |F_2||L|=|F_2||F_3||\S_3|=|K| $, so $ F_2\cap L=1 $. Thus, $ K=F_2{:}L$ is a semidirect product, which satisfies\,: \[\mbox{$ \C_L(F_2)=F_3 $, and $ K/\C_L(F_2)=\ov K=F_2{:}\S_3 $ as given in Table \ref{tab-ovK},}\] so that, $ K $ is determined once $ F_2 $, $ L $ and $ F_3 $ are all determined.

Note that, $ L $ has a normal subgroup $ F_3 $ such that $ L/F_3\cong\S_3 $. It follows that the Sylow subgroup $ L_3 $ of $ L $ is normal, and $L=L_3{:}\ZZ_2$. Thus, $ L $ is one of the groups given in Corollary \ref{lem:3gp}~(1), by checking which we conclude that, there are six possibilities for the pair $  (L,\,F_3) $ as follows, where $ \ell\ge1 $\,:
\[\mbox{$ (\D_{2.3^{\ell}},\,\ZZ_{3^{\ell-1}})$,\ \ $(\ZZ_{3^{\ell}}\times\D_{6},\, \ZZ_{3^{\ell}}),\ \ (\D_{2.3^{\ell}}\times\ZZ_3,\,\ZZ_{3^{\ell-1}}\times\ZZ_{3})$,}\]
\[\mbox{$((\ZZ_{3^{\ell}}\times\ZZ_3){:}\ZZ_2,\,\ZZ_{3^{\ell}}),\ \ ((\ZZ_{3^{\ell}}\times\ZZ_3){:}\ZZ_2,\,\ZZ_{3^{\ell-1}}\times\ZZ_{3})$ \,and\, $(\ZZ_{3^{\ell+1}}{:}\ZZ_6,\,\ZZ_{3^{\ell}}\times\ZZ_3)$.}\]

(1) Suppose $ F_2=\ZZ_2^2 $, so $ K=\ZZ_2^2{:}L $ satisfies that $ K/F_3=\ZZ_2^2{:}\S_3\cong\S_4 $. Then, respectively for the six possibilities, we obtain the following groups, which are collected in Table \ref{tab-main} as $ X $\,:
\[\mbox{$\ZZ_2^2{:}\D_{2.3^{\ell}}$, \ \ $\ZZ_{3^{\ell}}
\times \S_4,\ \ (\ZZ_2^2{:}\D_{2.3^{\ell}})\times\ZZ_3,$}\]
\[\mbox{$(\ZZ_{3^{\ell}}\times \A_4){:}\ZZ_2,\ \ ((\ZZ_2^2{:}\ZZ_{3^{\ell}})\times\ZZ_{3}){:}\ZZ_2$ \,and\, $\ZZ_2^2 {:}(\ZZ_{3^{\ell+1}}{:}\ZZ_6)$.}\]

(2) Suppose $ F_2=\ZZ_2^3 $, so $ K=\ZZ_2^3{:}L$ satisfies that $ K/F_3= \ZZ_2\times(\ZZ_2^2{:}\S_3) $. That is, $ L $ centralizes some $ \ZZ_2 $, and normalizes its complement $ \ZZ_2^2 $. Thus,  $K=\ZZ_2\times(\ZZ_2^2{:}L) $, where $ \ZZ_2^2{:}L=X $ is given above.

(3) Suppose $ F_2=\Q_8 $, so $ K=\Q_8{:}L $ satisfies that $ K/F_3=\Q_8{:}\S_3\cong\GL_2(3) $. Then, similar to (1), we obtain the following groups , which are collected in Table \ref{tab-main} as $ Y $\,:
\[\mbox{$\Q_8{:}\D_{2.3^{\ell}}$,\ \  $\ZZ_{3^{\ell}}\times \GL_2(3)$, \ \ $ (\Q_8{:}\D_{2.3^{\ell}})\times\ZZ_3$,}\]
\[\mbox{$(\ZZ_{3^{\ell}}\times (\Q_8{:}\ZZ_3)){:}\ZZ_2$,\ \
$((\Q_8{:}\ZZ_{3^{\ell}})\times\ZZ_{3}){:}\ZZ_2$ \,and\, $\Q_8 {:}(\ZZ_{3^{\ell+1}}{:}\ZZ_6)$}.\]

(4) Suppose $ F_2=\ZZ_4\circ \Q_8$, so $ K=(\ZZ_4\circ\Q_8){:}L $ satisfies that $ K/F_3=\ZZ_4\circ(\Q_8{:}\S_3) $. That is, $ L $ centralizes $ \ZZ_4 $, and normalizes $ \Q_8 $. Thus, $ \ZZ_4 $ and $ \Q_8{:}L=Y $ as given above are two normal subgroups of $ K $ commuting with each other. Further, $ \ZZ_4Y=K $, so $ |\ZZ_4\cap Y|=2 $, and $ K=\ZZ_4\circ_{\ZZ_2} Y $. Note that, $ \ZZ_4\circ_{\ZZ_2} Y $ cannot be simply written as $ \ZZ_4\circ Y $,  as we might have $ \Z(Y)>\ZZ_2 $.
\end{proof}

\begin{lemma}\label{III-1}
	If it is in case {\bf(III)}, then $ p=3 $, and $ (F_2,\O_2(\ov K))=(\ZZ_2^2,\ZZ_2^3)$ or $ (\Q_8,\ZZ_4 \circ \Q_8) $.
\end{lemma}

\begin{proof}
	
	In case {\bf (III)}, we have $ F_2\lhd\ov K=\O_2(\ov K){:}H $ and $ F_2<\O_2(\ov K) $.
	
	If $ p=7 $, then $ \ov K=\ZZ_2^3{:}\ZZ_7 $ has $ \O_2(\ov K)=\ZZ_2^3 $ as a minimal normal subgroup, and so $ F_2=1 $. That is, $ F=F_7 $, and $$ \ZZ_2^3{:}\ZZ_7=\ov K=K/F_7\lesssim\Out(F_7).$$ It is impossible, as by Lemma \ref{p-gp}, $ |\Aut(F_7)|$ being divisible by 8 only happens when $ F_7=\ZZ_7^2 $, $\Out(F_7)=\GL_2(7) $. However, $ \GL_2(7) $ has no subgroup isomorphic to $ \ZZ_2^3 $.
	
	Thus, $ p=3 $, $ H\in\{\ZZ_3,\S_3\} $, and $ \ov K=\O_2(\ov K){:}H $ is one of the groups in Table \ref{tab-ovK}\,:
	\[\ZZ_2^2{:}H,\ \ZZ_2\times(\ZZ_2^2{:}H),\ \Q_8{:}H,\ \ZZ_4\circ(\Q_8{:}H).\] Note that, $ \ZZ_2^2{:}H $ has $ \ZZ_2^2 $ as a minimal normal subgroup. Thus,
	\begin{itemize}
		\item[-] if $\ov K=\ZZ_2^2{:}H $, then $ F_2=1 $\,;
		\item[-] if $ \ov K=\ZZ_2\times(\ZZ_2^2{:}H)=\l a\r\times(\l b,c\r{:}H) $, then $ F_2$ is one of $1$, $\l a\r=\ZZ_2$, $\l b,c\r=\ZZ_2^2$.
	\end{itemize}
    Also note that, $ \Q_8{:}H $ only has the center $ \ZZ_2 $ and $ \Q_8 $ as normal $ 2 $-subgroups. Thus,
    \begin{itemize}
    	\item[-] if $ \ov K=\Q_8{:}H $, then $ F_2=1 $ or $ \ZZ_2 $\,;
    	\item[-] if $ \ov K=\ZZ_4\circ(\Q_8{:}H) $, then either $ F_2\le\Q_8$, so $ F_2=1, \ZZ_2, \Q_8 $\,; or $ F_2\Q_8=\ZZ_4\circ\Q_8 $, implying that $  F_2=\ZZ_4 $.
    \end{itemize}
	
\noindent In conclusion, $ (F_2,\O_2(\ov K)) $ with $ F_2<\O_2(\ov K) $ is one of the pairs listed as follows\,:
	\begin{itemize}
		\item[(a)]
		$(1,\ZZ_2^2)$, $(1,\ZZ_2^3)$, $(\ZZ_2,\ZZ_2^3)$,  $(\ZZ_2,\Q_8)$, $(\ZZ_2,\ZZ_4\circ\Q_8)$, $(\ZZ_4,\ZZ_4\circ\Q_8)$\,;\vs
		\item[(b)] $(1,\Q_8)$, $ (1,\ZZ_4\circ \Q_8) $\,;\vs
		\item[(c)] $(\ZZ_2^2,\ZZ_2^3)$, $ (\Q_8,\ZZ_4 \circ \Q_8) $.
	\end{itemize}

We thus need to exclude the pairs in (a) and (b). For the pairs in (a), we have\,:\[\mbox{$\ov K/F_2=(\O_2(\ov K){:}H)/F_2=\ZZ_2^2{:}H $ or $ \ZZ_2\times(\ZZ_2^2{:}H), $}\] which contains a subgroup $ \ZZ_2^2{:}\ZZ_3\cong\A_4 $. Meanwhile, $ F=F_3\times F_2 $, so $$ K/F\cong\ov K/F_2 \lesssim\Out(F)=\Out(F_3)\times\Out(F_2),$$ where $ \Out(F_2)=1 $ or $\ZZ_2$. Thus, $ \A_4 $ is embedded into either $\Out(F_3)$ or $ \Out(F_3)\times\ZZ_2 $. For the latter, as $ \A_4\cap\ZZ_2\lhd\A_4 $, $ \A_4\cap\ZZ_2=1 $, so $ \A_4 $ is also embedded into $ \Out(F_3) $. We thus get a contradiction by Corollary \ref{lem:3gp}~(3). Hence this case cannot occur.\vs

(b.1) Suppose $ (F_2,\O_2(\ov K))=(1,\Q_8)$. Then, $F=F_3$, $K/F_3=\Q_8{:}H\lesssim\Out(F_3)$. By lemma \ref{p-gp}, $ |\Out(F_3)| $ being divisible by $ 8 $ only happens when $$ F_3=\ZZ_3^2,\ \ \Out(F_3)=\GL_2(3)\cong\Q_8{:}\S_3. $$ Since $ H=\ZZ_3$ or $ \S_3 $, $ K_3>F_3 $, and in particular, $ K_3 $ acts on $ F_3 $ non-trivially, so $ K_3$ is non-abelian as given in Lemma \ref{p-gp}~(3), written explicitly as \[\mbox{$ K_3=\l a\r{:}\l b\r=\ZZ_9{:}\ZZ_3 $, with $a^b=a^4$, $b^a=a^6b$.}\] By simple calculation, any element lying in $ K_3\setminus\l a^3,b\r $ has order $ 9 $, so $ F_3=\l a^3\r\times\l b\r $.
	
Let $ c $ be the central involution of $ K_2=(\Q_8{:}H)_2=\Q_8$ or $\Q_8{:}\ZZ_2 $. Then $ \ov c $ is the central involution of $ \ov K\lesssim\GL_2(3) $, so $ \ov c $ acts on $ F_3 $ by mapping each $ x $ to $ x^{-1} $. That is, for any for any $ x\in F_3 $, we have
\begin{center}
	$ x^{\ov c}=x^c=x^{-1} $, so $ c^x=cx^2 $.
\end{center}
Meanwhile, $ \ov c\in\Z(\ov K) $, $ \ov c^{\ov a}=\ov c $, and so $ c^a=cy $ with some $ y=a^{3i}b^j\in F_3 $. Thus, for $ a^3\in F_3 $, we have \[c^{a^3}=(ca^{3i}b^j)^{a^2}=(ca^{(6i+6j)}b^{2j})^a=ca^{(9i+18j)}b^{3j}=c\ne ca^6,\] a contradiction. Hence this case cannot occur.\vs
	
(b.2) Suppose $(F_2, \O_2(\ov K))=(1,\ZZ_4\circ \Q_8)$. Then, $ F=F_3 $, $\ov K=\ZZ_4\circ(\Q_8{:}H) \lesssim\Out(F_3)$. Similar to (b.1), $ |\Out(F_3)| $ being divisible by $ 8 $ implies $ F_3=\ZZ_3^2 $, $ \Out(F_3)=\GL_2(3)$. By comparing the orders, we have $ H=\ZZ_3 $ and $ \ZZ_4\circ(\Q_8{:}\ZZ_3)=\GL_2(3) $, which is impossible as $ \Z(\GL_2(3))=\ZZ_2\ne\ZZ_4 $. Hence this case cannot occur.
\end{proof}

\begin{lemma}\label{III-2}
	If $ (F_2,\O_2(\ov K))=(\ZZ_2^2,\ZZ_2^3)$, then $ K = W $ is given in {\rm Table \ref{tab-main}} (the left side), homomorphic to $\A_4$ or $\S_4$.
\end{lemma}

\begin{proof}
The proof is similar to Lemma \ref{II}. In this case, we have\,: \[K=F_3.\ov K=F_3.(\ZZ_2\times(\ZZ_2^2{:}H))=F_3.(\ZZ_2^2{:}(\ZZ_2\times H)), \mbox{\ where\ } F_2=\ZZ_2^2.\]
Let $ L\le K $ the full preimage of $\ZZ_2\times H\le \ov K$. Then, $ \l F_2,L\r=F_2L $ is the full preimage of $\ov K$, so $ F_2L=K $. Meanwhile, $ |F_2||L|=|F_2||F_3|\cdot 2|H|=|K| $, so $ F_2\cap L=1 $. Thus, $ K=F_2{:}L=\ZZ_2^2{:}L$ is a semidirect product, which satisfies\,: \[\mbox{$ \C_L(\ZZ_2^2)=F_3{:}\ZZ_2 $, and $ K/\C_L(\ZZ_2^2)\cong\ov K/\ZZ_2=\ZZ_2^2{:}H\cong \A_4$ or $ \S_4 $,}\] so that, $ K $ is determined once the pair $(L,F_3) $ is determined.

(1) Suppose $ H=\ZZ_3 $. Then, $ L $ has a normal subgroup $ F_3 $ such that $ L/F_3=\ZZ_6 $. It follows that the Sylow subgroup $ L_3 $ of $ L $ is normal, and $L=L_3{:}\ZZ_2$. Further, $ L\ne L_3\times\ZZ_2 $; otherwise we can get a normal 2-subgroup of $ K $ larger than $ F_2 $. Thus, $ L $ is one of the groups given in Corollary \ref{lem:3gp}~(1), by checking which we conclude that, there are three possibilities for the pair $ (L,F_3) $ as follows, where $ \ell\ge 1 $\,: \[\mbox{$(\D_{2.3^{\ell}}\times\ZZ_3,\,\ZZ_{3^{\ell}})$, \  $(\ZZ_{3^{\ell}}\times\D_{6},\,\ZZ_{3^{\ell-1}}\times\ZZ_3 )$, \  $(\ZZ_{3^{\ell+1}}{:}\ZZ_6,\,\ZZ_{3^{\ell+1}})$,}\] which, respectively, provide groups $ K=\ZZ_2^2{:}L $ with $ K/({F_3{:}\ZZ_2})\cong\ZZ_2^2{:}\ZZ_3= \A_4 $ as follows\,: \[\mbox{$\D_{2.3^{\ell}}\times \A_4$, \ $  (\ZZ_2^2{:}{\ZZ_{3^{\ell}}})\times \D_6$, \ $(\ZZ_2^2\times\D_{2.3^{\ell+1}}){:}\ZZ_3$.}\]

(2) Suppose $ H=\S_3 $. Then, $ L $ has a normal subgroup $ F_3 $ such that $ L/F_3=\ZZ_2\times\S_3 $, and so $ L=L_3{:}(\ZZ_2\times\ZZ_2)$. Further, $ \C_{\ZZ_2\times\ZZ_2}(L_3) =1$; otherwise we can get a normal 2-subgroup of $ K $ larger than $ F_2 $. Thus, by Corollary \ref{lem:3gp}~(2), $L=\D_{2.3^{\ell}}\times\D_6 $ with $ \ell\ge 1 $, and $F_3=\ZZ_{3^{\ell}}$ or $ \ZZ_{3^{\ell-1}}\times\ZZ_3$. This provides groups $ K=\ZZ_2^2{:}L $ with $ K/({F_3{:}\ZZ_2})\cong\ZZ_2^2{:}\S_3= \S_4 $ as follows\,: \[\mbox{$\D_{2.3^{\ell}}\times\S_4$\, or\, $ (\ZZ_2^2{:}\D_{2.3^{\ell}})\times\D_6$.}\]
\end{proof}

\begin{lemma}\label{III-3}
	If $ (F_2,\O_2(\ov K))=(\Q_8,\ZZ_4 \circ \Q_8)$, then $ K = W $ is given in {\rm Table \ref{tab-main}} (the right side), homomorphic to $\ZZ_4\circ\SL_2(3)$ or $\ZZ_4\circ\GL_2(3)$.
\end{lemma}

\begin{proof}

In this case, we have\,: \[K=F_3.((\ZZ_4\circ\Q_8){:}H)= F_3.(\ZZ_4\circ(\Q_8{:}H)),\] so $ K $ is homomorphic to $\ZZ_4\circ\SL_2(3)$ or $\ZZ_4\circ\GL_2(3)$, respectively for $ H=\ZZ_3 $ or $ \S_3 $.

Let $ L\leqslant K $ be the full preimage of $ \ZZ_4\times H\leqslant\ov K$. Then, $ \l F_2,L\r =F_2L$ is the full preimage of $ \ov K $, so $K=F_2L $.

(1) Suppose $ H=\ZZ_3 $. Then, $ L $ has a normal subgroup $ F_3 $ with $ L/F_3=\ZZ_{12} $. It follows that the Sylow subgroup $ L_3 $ is normal, and $ L=L_3{:}\ZZ_4$. Let $ x\in\C_{\ZZ_4}(L_3)$, so $ x\in\Z(L) $, and $ F_2\l x\r $ is a normal 2-subgroup of $ K $. It implies $ |x|<4 $, and so $ \C_{\ZZ_4}(L_3)= \ZZ_4\cap \Q_8=\ZZ_2$. That is, $ L/\C_{\ZZ_4}(L_3)=L{:}\ZZ_2 $ is one of the groups given in Corollary \ref{lem:3gp}~(1), by checking which we conclude that, there are three possibilities for the pair $ (L,F_3) $ as follows, where $ \ell\ge1 $.

(1.1) $(L,\,F_3)=((\ZZ_{3^{\ell}}{:}\ZZ_4)\times\ZZ_3,\,\ZZ_{3^{\ell}})$. Let $ M=\ZZ_{3^{\ell}}{:}\ZZ_4 $, $ N=\l F_2,\ZZ_3\r=\Q_8{:}\ZZ_3 $. Then, $ M $, $ N $ are two normal subgroups of $ K $, which generate $ K $ and commutes with each other. In particular, $ |M\cap N|=2$, so $M\cap N=\ZZ_2=\ZZ_4\cap\Q_8=\ZZ(N) $, and \[ K=M\circ N=(\ZZ_{3^{\ell}}{:}\ZZ_4)\circ(\Q_8{:}\ZZ_3).\]

(1.2) $ (L,\,F_3)=(\ZZ_{3^{\ell}}\times(\ZZ_{3}{:}\ZZ_4),\,\ZZ_{3^{\ell-1}}\times\ZZ_3)$. Similar to (1.1),  we have\,: \[K= (\Q_8{:}\ZZ_{3^\ell})\circ(\ZZ_{3}{:}\ZZ_4).\]

(1.3) $ (L,F_3)=(\ZZ_{3^{\ell+1}}{:}\ZZ_{12},\ZZ_{3^{\ell+1}})$. Let $ M=\ZZ_{3^{\ell+1}}{:}\ZZ_4 $. Then, $ F_2=\Q_8 $ and $ M $ are two normal subgroups of $ K $ commuting with each other. Moreover, $ \Q_8M $ is normal in $ K $, with a complementary subgroup $ \ZZ_3\le\ZZ_{12} $, so $ K=(\Q_8M){:}\ZZ_3 $. In particular, $ | \Q_8\cap M|=2 $, $ \Q_8\cap M=\ZZ_2=\Q_8\cap\ZZ_4=\Z(\Q_8) $, so $ \Q_8M=\Q_8\circ M $, and \[K=(\Q_8\circ (\ZZ_{3^{\ell+1}}{:}\ZZ_4)){:}\ZZ_3.\]

(2) Suppose $ H=\S_3 $. Then, similar to (1), it is shown that, $ L=L_3{:}(\ZZ_4\times\ZZ_2)$ with $ \C_{\ZZ_4\times\ZZ_2}(L_3)= \ZZ_4\cap F_2=\ZZ_2$, so by Corollary \ref{lem:3gp}~(2), $ L/\C_{\ZZ_4\times\ZZ_2}(L_3)=L{:}(\ZZ_2\times\ZZ_2)=\D_{2.3^{\ell}}\times\D_6,$ where $ \ell\ge1 $. Note that, $ L/F_3=\ZZ_4\times\S_3 $. We thus have two possibilities\,: \[ (L,\,F_3)=((\ZZ_{3^{\ell}}{:}\ZZ_4)\times\D_6,\,\ZZ_{3^{\ell}}) \mbox{\ or\ } (\D_{2.3^{\ell}}\times\D_6,\,\ZZ_{3^{\ell-1}}\times\ZZ_3).\] This, similar to (1.1), (1.2), provides groups $ K $ as follows\,: \[\mbox{$ (\ZZ_{3^{\ell}}{:}\ZZ_4)\circ(\Q_8{:}\S_3)$ \,or\, $ (\Q_8{:}\D_{2.3^{\ell}})\circ(\ZZ_3{:}\ZZ_4)$.}\]
\end{proof}

\section{Proof of Theorem~\ref{thm:arctransmap-gps}}\label{sec:thm-2}

In this section, we focus on the arc-transitive automorphism groups of maps. As the following lemma shows, the class of such groups is nearly closed under taking quotients.

\begin{lemma}\label{lem:arctrans}
Let $G$ be a finite group, with $\ov G$ a non-identity factor group. Then:
\begin{itemize}
	\item[\rm(1)] If $ G $ has a regular triple (resp. reversing triple), then either $ \ov G $ has a regular
	triple (resp. reversing triple), or $ \ov G $ is $ \ZZ_2 $ or a dihedral group. In particular, $ \ov G $ is of even order.\vs
	\item[\rm(2)] If $ G $ has a rotary pair, then either $ \ov G $ has a rotary pair, or $ \ov G $ is cyclic. In particular, if $ \ov G $ is of odd order, then it is cyclic.
\end{itemize}
\end{lemma}

\begin{proof}
Let $(x,y,z)$ be a regular triple for $ G $. Then, $G=\l x,y,z\r$ and $ \l x,z\r=\ZZ_2^2 $, so $\ov G=\l\ov x,\ov y,\ov z\r $, $\ov x\,\ov z=\ov z\,\ov x$. Now, if $\ov x,\ov y,\ov z$ are involutions, and $ \ov x\ne\ov z $, then $ (\ov x,\ov y,\ov z) $ is a regular triple for $ \ov G $. Otherwise, either some of $\ov x,\ov y,\ov z$ equals the identity-element $ \ov 1 $, or $ \ov x=\ov z $, where, in both cases, $\ov G$ is generated by at most two involutions, so it is $ \ZZ_2 $ or dihedral. In particular, $ \ov G $ is always of even order.

It is similar to prove the remaining part of this lemma.
\end{proof}

\begin{lemma}\label{not-Q:S3}
Let $G=\GL(2,3)$. Then neither $G$ nor $ \ZZ_4\circ G $ has a regular triple.
\end{lemma}

\begin{proof}
Suppose that $G\cong\Q_8{:}\S_3$ has a regular triple $(x,y,z)$.
Then, $ \l x,z\r=\ZZ_2^2 $, so $\{x,z,xz\}\cap\Q_8\not=\emptyset$, namely, one of $x,z,xz$ is the central involution $ -1 $ of $G$. It follows that $\ov G=G/\l-1\r=\l \ov x,\ov y\r$ or $\l\ov y,\ov z\r$ is dihedral, and this contradicts $\ov G\cong\ZZ_2^2{:}\S_3=\S_4$.

Suppose that $K=\ZZ_4\circ G\cong\ZZ_4\circ(\Q_8{:}\S_3)$ has a regular triple $(x,y,z)$.
Write $ K $ as \[ K=\l u\r\circ(\Q_8{:}(\l\s\r{:}\l\tau\r)),\] where $ |u|=4 $, $ \Q_8=\pm\{1,i,j,k\} $ with $ i^2=j^2=k^2=ijk=-1 $, and $ \s,\tau\in \Aut(\Q_8) $ act on $ \Q_8 $ as follows\,: \[ \s: i\rightarrow j\rightarrow k\, (\rightarrow i)\mbox{\,\ and\,\ }\tau: i\leftrightarrow -j, j\leftrightarrow -i, k\leftrightarrow -k .\]
Then, $ \l\s\r{:}\l\tau\r=\S_3 $ contains three involutions, i.e., $ \tau_k:=\tau $, and $ \tau_i $, $ \tau_j $ such that \[\tau_i: j\leftrightarrow -k, k\leftrightarrow -j, i\leftrightarrow -i \mbox{\,\ and\,\ } \tau_j: i\leftrightarrow -k, k\leftrightarrow -i, j\leftrightarrow -j.\]
By direct calculation, we conclude that the involutions contained in $ K $ are as follows\,: \[\{-1\}\cup\{ \pm u\lambda \}\cup \{\pm \tau_{\lambda}, \pm \lambda\tau_{\lambda}\}, \mbox{\ where\ } \lambda\in\{i,j,k\}.\]

\noindent {\bf Claim.} The only involutions commutes with $ u\lambda $ (or $-u\lambda $) are $ \pm u\lambda$ and $ -1 $.

In fact, for any $ \lambda'\in\{i,j,k\} $ with $ \lambda\ne\lambda' $, we have\,:
\begin{center}
	$ \lambda^{\lambda'}= \lambda^{\tau_\lambda}=\lambda^{\lambda\tau_\lambda}=-\lambda $ \ and\
	$ \lambda^{\tau_{\lambda'}}=-\lambda'',\ \lambda^{\lambda'\tau_{\lambda'}}=-\lambda^{\tau_{\lambda'}}=\lambda'' $,
\end{center}
where $\lambda''\in\{i,j,k\}\setminus\{\lambda,\lambda'\}.$
As $ u $ and $ -1 $ lie in the center of $ K $, we then have\,:
\[(\pm u\lambda)^{\pm u\lambda'}=\pm u(\lambda^{\pm u\lambda'})=\pm u(\lambda^{\lambda'})=\mp u\lambda,\]
 \[(\pm u\lambda)^{\pm \tau_\lambda}=\pm u (\lambda^{\pm \tau_\lambda})=\mp u\lambda,\ (\pm u\lambda)^{\pm \lambda\tau_\lambda}=\pm u (\lambda^{\pm \lambda\tau_\lambda})=\mp u\lambda,\]
 \[(\pm u\lambda)^{\pm \tau_{\lambda'}}=\pm u(\lambda^{\pm \tau_{\lambda'}})=\mp u\lambda'' \mbox{\ and\ } (\pm u\lambda)^{\pm \lambda'\tau_{\lambda'}}=\pm u(\lambda^{\pm \lambda'\tau_{\lambda'}})=\pm u\lambda''.\]
We have therefore confirmed the claim.

Let $ H=\l u,i,j,k\r\leqslant K $. Then, $ H\cong\ZZ_4\circ\Q_8 $ is normal in $ K $, $ K/H\cong \S_3 $, and the involutions contained in $ H $ lie in $ \{-1\}\cup\{ \pm u\lambda \} $. Now, $ (x,y,z) $ is a regular triple for $ G $ such that $ \l x,z\r=\ZZ_2^2 $, so $ \l x,z\r\cap H >1 $, that is, \[ \{x,z,xz\}\cap (\{-1\}\cup\{ \pm u\lambda \})\not=\emptyset.\]
By the above claim, one of $x,z,xz$ has to be $-1$. It follows that $\ov K=K/\l-1\r=\l \ov x,\ov y\r$ or $\l\ov y,\ov z\r$ is dihedral, and this contradicts $\ov K\cong\ZZ_2\times(\ZZ_2^2{:}\S_3)=\ZZ_2\times\S_4$.
\end{proof}

\begin{lemma}\label{not-332}
	Let $ G=(\ZZ_{3^\ell}\times\ZZ_3){:}\ZZ_2=\l a,b\r{:}\l c\r $, with $ a^c=a^{-1} $, $ b^c=b^{-1} $. Then $ G $ has no rotary pair.
\end{lemma}

\begin{proof}
	Let $ (\a,z) $ be a rotary pair for $ G $. Then $ G=\l \a,z\r $ and $ |z|=2 $. Since $ G $ is not dihedral, $ |\a|\ne 2 $, so $ \a\in\l a,b\r $ and $ \l\a\r $ is normal in $ G $. It follows that $ \ov G=G/\l\a\r=\l\ov z\r=\ZZ_2 $, not possible.
\end{proof}

Now, we are ready to prove Theorem~\ref{thm:arctransmap-gps}.

{\bf Proof of Theorem~\ref{thm:arctransmap-gps}\,:}

Let $\calM$ be a map with square-free Euler characteristic, and let $G\le\Aut\calM$ be a solvable arc-transitive automorphism group of $ \calM $.
By Lemmas~\ref{Sylow-metac} and \ref{class-closed}, the group $ G $ satisfies Hypothesis \ref{hypo-0}, so by Theorem~\ref{thm:sfEul-gps}, we have\,: \[G=H{:}K=(A{:}B){:}K,\] where $A$ is abelian, $ B $ is nilpotent and $ A\cap\Z(A{:}B)=1 $. We thus only need to present $ K $. \vs

Recall that $ \calM $ as a $ G $-arc-transitive map satisfies one of the following properties\,:
$ G $-regular\,; $ G $-vertex-reversing\,; $ G $-vertex-rotary\,; respectively, in each
case, the group $ G $ has\,: a regular triple\,; a reversing triple\,; a rotary pair. In particular, $ G $ is of even order, so $ K \ne 1 $.\vs

(a) Suppose $K=G_2$. By Lemma \ref{lem:arctrans}, either $ K $ has a regular triple/reversing triple/rotary pair as $ G $ does, and then $ K $ is one of the 2-groups given in Proposition \ref{2-gp-arc}; or $ K $ is $ \ZZ_2 $, dihedral $ \D_{2^\ell} $ or cyclic $ \ZZ_{2^\ell} $ ($ \ell\ge 2 $), while $ \D_{2^\ell} $ and $ \ZZ_{2^\ell} $ already appear in Proposition \ref{2-gp-arc}.

This give the Theorem \ref{thm:arctransmap-gps}~(1).\vs

(b) Suppose that $ K $ is one of the groups given in Theorem \ref{thm:sfEul-gps} (ii) and (iv)\,:
\[\ZZ_2^2{:}G_3,\ \ZZ_2\times(\ZZ_2^2{:}G_3),\ \Q_8{:}G_3,\ \ZZ_4\circ_{\ZZ_2}(\Q_8{:}G_3),\ \ZZ_2^3{:}G_7,\ \ZZ_2^3{:}(G_7{:}G_3).\] Then, $ G $ has an odd factor group $ G_3 $, $ G_7 $ or $ G_7{:}G_3 $. By Lemma \ref{lem:arctrans}, $ G $ has neither regular triples nor reversing triples, so $ G $ has a rotary pair, with each odd factor group cyclic, namely, $ G_3=\ZZ_{3^\ell} $, $ G_7=\ZZ_{7^\ell} $ and $ K\ne\ZZ_2^3{:}(G_7{:}G_3) $. Thus, $ K $ is one of\,: \[\mbox{$\ZZ_{2}^2{:}\ZZ_{3^\ell}$,\, $ \ZZ_2\times(\ZZ_{2}^2{:}\ZZ_{3^\ell} )$,\, $\ZZ_4\circ_{\ZZ_2}(\Q_8{:}\ZZ_{3^\ell})$, \,$ \ZZ_2^3{:}\ZZ_{7^\ell}, $}\] where $\Q_8{:}\ZZ_{3^\ell} $ is excluded, because it has the only involution lying in the center, and consequently, it does not have a rotary pair.

This gives the first half of Theorem \ref{thm:arctransmap-gps}~(4).\vs

(c) At last, suppose that $K$ is one of the groups given in Theorem \ref{thm:sfEul-gps} (iii)\,: $ X $, $ \ZZ_2\times X $, $ Y $, $ \ZZ_4\circ_{\ZZ_2} Y $ and $ W $, where $ X,Y,W $ are given in {\rm Table \ref{tab-main}}.

If $ \calM $ is $ G $-vertex-reversing, then $ G $ has a reversing triple, with each factor group of even order by Lemma \ref{lem:arctrans}. Note that, the groups $ X,Y $ located in rows $2,3,6$ of Table \ref{tab-main} have an odd factor group $ \ZZ_3 $ or $ \ZZ_{3^\ell} $, so does the groups $ W $ located in rows $1,2,5$. Thus, $ X,Y,W $ can only taken from the others, namely,

$X\in\{\ZZ_2^2{:}\D_{2.3^{\ell}},\,(\ZZ_{3^{\ell}}\times \A_4){:}\ZZ_2,\,((\ZZ_2^2{:}\ZZ_{3^{\ell}})\times\ZZ_{3}){:}\ZZ_2\},$

$Y\in\{ \Q_8{:}\D_{2.3^{\ell}},\, (\ZZ_{3^{\ell}}\times (\Q_8{:}\ZZ_3)){:}\ZZ_2,\, (\ZZ_{3}\times (\Q_8{:}\ZZ_{3^{\ell}})){:}\ZZ_2\}$, and

$W\in\{\D_{2.3^{\ell}}\times\S_4,\, (\ZZ_2^2{:}\D_{2.3^{\ell}})\times\D_6,\, (\ZZ_{3^{\ell}}{:}\ZZ_4)\circ(\Q_8{:}\S_3),\,
(\Q_8{:}\D_{2.3^{\ell}})\circ(\ZZ_3{:}\ZZ_4)\}$.\vs

If $ \calM $ is $ G $-regular, then $ G $ has a regular triple, so it has a reversing triple, and $ K $ is one of the above groups. Furthermore, by Lemma \ref{not-Q:S3}, the groups which have $\GL_2(3)$ or $ \ZZ_4\circ \GL_2(3)$ as a factor group should be excluded. Hence, $ G\ne Y $, $ \ZZ_4\circ_{\ZZ_2} Y $, $(\ZZ_{3^{\ell}}{:}\ZZ_4)\circ(\Q_8{:}\S_3)$ or $(\Q_8{:}\D_{2.3^{\ell}})\circ(\ZZ_3{:}\ZZ_4)$.

This gives Theorem \ref{thm:arctransmap-gps}~(2) and (3).\vs

(c.3) If $ \calM $ is $ G $-vertex-rotary, then $ G $ has a rotary pair, so by Lemma \ref{not-332}, the following four groups, which have $ (\ZZ_{3^\ell}\times\ZZ_3){:}\ZZ_2 $ as a factor group, should be excluded\,:
\[\mbox{$(\ZZ_{3^{\ell}}\times \A_4){:}\ZZ_2$,\, $(\ZZ_{3}\times (\ZZ_2^2{:}\ZZ_{3^{\ell}})){:}\ZZ_2$,\, $(\ZZ_{3^{\ell}}\times (\Q_8{:}\ZZ_3)){:}\ZZ_2$,\, $(\ZZ_{3}\times (\Q_8{:}\ZZ_{3^{\ell}})){:}\ZZ_2$.}\]

Finally, this completes Theorem \ref{thm:arctransmap-gps}~(4).\qed

\section*{Declarations}
\begin{center}
	{Funding and/or Conflicts of interests/Competing interests}
\end{center}

The author(s) declares that there is no any financial or personal
relationship with other people or organizations not mentioned that can
inappropriately influence the work.

\end{document}